%% file: Coherence.tex
\documentclass[11pt]{amsart}
\usepackage{graphicx, mathabx, amssymb,amsfonts,amsmath,amsthm,newlfont}
\usepackage{epsfig,url}
\usepackage{enumerate,enumitem}
\usepackage[colorlinks=true,linkcolor=red,citecolor=blue]{hyperref}
\usepackage[dvipsnames]{xcolor}
\usepackage{color}

\input{aux/commands}

\title{Topological proofs of categorical coherence}

\author{Pierre-Louis Curien}
\address{IRIF, CNRS, Universit\'e Paris Cit\'e and $\pi r^2$ team, Inria, France.}
\email{curien@irif.fr}

\author{Guillaume Laplante-Anfossi}
\address{School of Mathematics and Statistics, The University of Melbourne, Victoria, Australia.}
\email{guillaume.laplanteanfossi@unimelb.edu.au}

\date{\today}



\dedicatory{"We shall construct $KP_n$, as a CW-complex, in Section 2 and show that it is an $(n-1)$-ball. This gives an instant one-step proof of MacLane's theorem in full generality."  \\ --  Mikhail M. Kapranov}

\thanks{The second author was supported by the Andrew Sisson Fund and the Australian Research Council Future Fellowship FT210100256.}

\begin{document}

\begin{abstract}
We give a short topological proof of coherence for categorified non-symmetric operads by using the fact that the diagrams involved form the 1-skeleton of simply connected CW complexes. 
We also obtain a ``one-step'' topological proof of Mac Lane's coherence theorem for symmetric monoidal categories, as suggested by Kapranov in 1993.  
Our analysis is based on a notion of combinatorial homotopy, which we further study in the special case of polyhedral complexes, leading to a second geometrical proof of coherence which is very close to Mac Lane's original argument.  
We use Morse theory to show that this second method is (strictly) less general than the first.
We provide a detailed analysis of how  both methods allow us to deduce these two categorical coherence results and discuss possible generalizations to higher categories. 
\end{abstract}

\maketitle

\setcounter{tocdepth}{1}

\input{sec/introduction}

\input{sec/polycoherence}
\input{sec/catoperads}

\input{sec/monoidal}
\input{sec/applications}

\bigskip

\emph{Acknowledgements.}   
We would like to thank all the participants of the Roberta Seminar held on September 30th, 2020 in Paris for planting the seeds of the present note.  
We would like to thank Andrea Bianchi for enlightening discussions, for pointing out mistakes in a first draft, and for reminding us of the Seifert--Van Kampen theorem. 
Finally, we are grateful to Zoran Petri{\'c} for precious discussions on weak Cat-operads, to Marcy Robertson and to the anonymous referee for helpful comments which led to substantial improvements of the paper.
The first author wishes to thank the Mathematisches Forschungsinstitut Oberwolfach for several visits as a Research Fellow in 2020 and 2021, thanks to which much of the work presented here, as well as in the companion paper \cite{CLA24}, could take off.
The second author would like to thank the Copenhagen Center for Geometry and Topology as well as the Max Planck Institute for Mathematics in Bonn where part of this work was carried out.

\bibliographystyle{amsalpha}

\bibliography{Coherence}

\end{document}

%% file: aux/commands.tex
\usepackage[noabbrev,capitalize]{cleveref}

\usepackage[all,2cell]{xy} \UseAllTwocells \SilentMatrices

\calclayout

\newtheorem{thm}{Theorem}[section]
\newtheorem{corollary}[thm]{Corollary}
\newtheorem{lemma}[thm]{Lemma}
\newtheorem{proposition}[thm]{Proposition}

\theoremstyle{definition}
\newtheorem{definition}[thm]{Definition}

\theoremstyle{remark}
\newtheorem{rem}[thm]{Remark}

\numberwithin{equation}{section}

\newcommand{\set}[1]{\left\{#1\right\}}

\newcommand{\frakh}{\mathfrak{h}}
\newcommand{\R}{\mathbb R}

\newcommand{\obj}{\mathrm{Ob}} 

\newcommand{\id}{\mathrm{id}}

\newcommand{\setc}[2]{\set{#1 \mid #2}}

\definecolor{darkblue}{rgb}{0,0,0.7} 
\newcommand{\darkblue}{\color{darkblue}} 
\newcommand{\defn}[1]{{\darkblue \emph{#1}}}

\newcommand{\so}{\mathrm{sc}} 
\newcommand{\sk}{\mathrm{sk}}

\newcommand{\calS}{\mathcal{S}}

\DeclareMathOperator{\oLk}{Lk^{+}}

\newcommand{\eqdef}{\mbox{\,\raisebox{0.2ex}{\scriptsize\ensuremath{\mathrm:}}\ensuremath{=}\,}} 
\newcommand{\defeq}{\mbox{~\ensuremath{=}\raisebox{0.2ex}{\scriptsize\ensuremath{\mathrm:}} }} 

\newcommand{\extu}[1]{[[#1]]}
\newcommand{\extd}[2]{[[#1]]^{#2}}

\usepackage{graphicx,calc}
\newlength\myheight
\newlength\mydepth
\settototalheight\myheight{Xygp}
\usepackage{tikz}
\usepackage{tikz-cd}
\usepackage{pgfplots}
\usepackage{pgfplotstable}
\tikzset{math3d/.style=
    {x= {(-0.353cm,-0.353cm)}, z={(0cm,1cm)},y={(1cm,0cm)}}}
\tikzset{JLL3d/.style=
    {x= {(0.4cm,-0.2cm)}, z={(0cm,1cm)},y={(-1cm,0cm)}}}
\usetikzlibrary{calc}
\usetikzlibrary{shapes,shapes.geometric,fit,positioning,calc,matrix}
\tikzset{
  optree/.style={scale=.5,thick,grow'=up,level distance=10mm,inner sep=1pt},
  comp/.style={draw=none,circle,fill,line width=0,inner sep=0pt},
  dot/.style={draw,circle,fill,inner sep=0pt,minimum width=3pt},
  circ/.style={draw,circle,inner sep=1pt,minimum width=4mm},
  emptycirc/.style={draw,circle,inner sep=1pt,minimum width=2mm},
  root/.style={level distance=10mm,inner sep=1pt},
  leaf/.style={draw=none,circle,fill,line width=0,inner sep=0pt},
  nodot/.style={draw,circle,inner sep=1pt},
}

\pgfplotsset{compat=1.12}

%% file: sec/introduction.tex

\section*{Introduction} 
\label{s:introduction}

The $n$-dimensional permuto-associahedron, a CW-complex whose faces are in bijection with parenthesized ordered partitions of $n+1$ letters, was first introduced by M. Kapranov in his study of higher dimensional Yang--Baxter equations, through the moduli spaces of curves $\overline{\mathcal{M}_{0,n+1}}(\R)$ and the solutions of the Knizhnik--Zamolodchikov equation \cite{kapranov1993}.
It was later realized as a convex polytope by V. Reiner and G. M. Ziegler \cite{reinerCoxeterassociahedra1994}, and more recently through the nested braid fan by F. Castillo and F. Liu in \cite{CastilloLiu21}.

The present study stems from a desire to understand the epigraph, taken from the introduction of \cite{kapranov1993}: what is the precise relationship between the permuto-associahedron and Mac Lane's coherence theorem for symmetric monoidal categories? 
We show that the \emph{simple connectedness} of the former implies the latter, thereby refining and proving Kapranov's claim (see \cref{thm:coherence-MacLane}). 

This is done through a general ``topological coherence theorem" which applies to any simply connected, regular CW complex (\cref{thm:top-coherence}).
Applying it to the operahedra, another family of polytopes which encodes categorified non-symmetric operads~\cite{DP15,curienSyntacticAspectsHypergraph2019a,laplante-anfossiDiagonalOperahedra2022a}, we obtain a ``one-step'' proof of the associated coherence theorem as well.

\smallskip
There is little price to pay, though. 
For both theorems, one needs to provide a precise bijective correspondence between the 1-skeleton (resp.\ the 2--cells) on the topological side, and canonical morphisms (resp.\ bifunctoriality, naturality, and applications of coherence conditions) on the categorical side (\cref{{bijections-Kapranov},prop:bijection-nestings}). 
Since the 2-skeleton of the permuto-associahedra corresponds to other basic canonical morphisms and coherence conditions than those of Mac Lane (hexagons and naturality of the involutive braiding on one hand versus dodecagons on the other hand), one needs to show that the two presentations are equivalent, which is non-trivial, see~\cref{rem:Kapranov-to-MacLane}.  
There is yet a third equivalent presentation (and hence another proof of coherence) due to D. Barali\'c, J. Ivanovi\'c and D. Petri\'c~\cite{baralicSimplePermutoassociahedron2019}, that matches the 2-skeleton of a different polytope, which unlike the permuto-associahedron is simple, see \cref{rem:simple-permutoassociahedron}.

\smallskip
We further investigate a topological incarnation of Mac Lane's original argument, in the spirit of rewriting theory.  
We study polyhedral complexes endowed with a generic orientation vector, or equivalently a Morse function in the sense of \cite{bestvinaMorseTheoryFiniteness1997}, whose $1$-skeletons naturally feature terminating and confluent rewriting systems (\cref{prop:terminating-confluent}).
We focus on the family of simply connected polyhedral complexes whose outgoing links are connected. 
The study of directed paths on their $1$-skeleton leads to a second general proof of coherence (\cref{p:second-proof}).
In particular, this second theorem can be applied to all polytopes, allowing us to give a second, ``rewriting-theoretic'' proof of both previously mentioned coherence results.  
In the case of operahedra, our rewriting proof simplifies the original proof of Do{\v s}en and Petri{\'c}~\cite{DP15}, see~\cref{rem:DPLA}.

It is worth noting that, while the above polyhedral complexes admit \emph{abstract} rewriting systems on their $1$-skeleton,
the family of operahedra (which includes the associahedra, encoding non-symmetric monoidal categories) further admits \emph{term} rewriting systems, which exhibit more structure and are the subject of a companion paper~\cite{CLA24}. In contrast, we shall argue that the abstract rewriting approach to \emph{symmetric} monoidal categories is not informative, see \cref{MacLane-Kapranov-Simple}.

Using Morse theory on affine cell complexes \cite{bestvinaMorseTheoryFiniteness1997}, we relate our two approaches by showing that the second is (strictly) less general than the first (\cref{lemma:outgoing-link}).
 
\smallskip
Our two general topological coherence theorems can be used to prove other categorical results where polytopes appear, such as coherence for monoidal functors between monoidal categories \cite{epsteinFunctorsTensoredCategories1966}, see \cref{sec:further}.
They also shed light on some statements in the literature, such as the proof of \cite[Prop.~3.9]{KapranovVoevodsky94}, see \cref{sec:higher}.
This all points towards further investigation of the relationship between $n$-categorical coherence and $n$-connectedness of appropriate spaces.
While topological proofs of $2$-categorical coherence already appeared in \cite{Gurski11}, higher dimensional results have been obtained recently by S. Barkan in the context of $\infty$-operads~\cite{barkanArityApproximationInfty2022}, for which the present results could well be the strict, $n=1$ case.

%% file: sec/polycoherence.tex

\section{Topological coherence} 
\label{s:polycoherence}


\subsection{Coherence \`a la Van Kampen}

Let $X$ be a regular CW complex, and let $X_k$, $k\geq 0$ denote its $k$-skeleton. 
For an edge $e$ of $X$, denote its attaching map $f_e : \mathbb{S}^0 \to X_0$.
Consider the category $\mathcal{A}(X)$ with set of objects $X_0$, and generating morphisms $\alpha_e: f_e(-1) \to f_e(1)$ and $\alpha_e^{-1}: f_e(1) \to f_e(-1)$ for each edge $e \in X_1$.
A \defn{combinatorial path} on $X$ is a composable sequence of $\alpha$ and $\alpha^{-1}$ morphisms (a \emph{word} in $\alpha$ and $\alpha^{-1}$).
Two combinatorial paths $\gamma, \gamma' \in \mathcal{A}(X)(x,y)$ with the same endpoints are said to be \defn{parallel}.

Let $A$ be a $2$-cell of $X$, let $f_A : \mathbb{S}^1 \to X_1$ be its attaching map, and $x\in X_0$ be a vertex  in the image of $f_A$. Then $f_A$ defines a morphism $\gamma_A \in \mathcal{A}(X)(x,x)$, given by the sequence of edges $e_1,\ldots,e_n$ in its image starting at~$x$ and respecting the anti-clockwise orientation of $\mathbb{S}^1$.
Here, one selects $\alpha_{e_i}$ if the orientation of $f_A$ restricted to $e_i$ agrees with the one of $f_{e_i}$, and $\alpha_{e_i}^{-1}$ otherwise.
Two parallel combinatorial paths $\gamma, \gamma'$ are said to be \defn{elementary combinatorially homotopic} if they differ exactly by a relation of the form $\alpha_e \alpha_e^{-1}=\id_{f_e(1)}$ or $\alpha_e^{-1} \alpha_e=\id_{f_e(-1)}$, or of the form
 $\gamma_A = \id_{x}$, for some $2$-cell $A$ and vertex $x$ as above.
That is, one can rewrite $\gamma$ into $\gamma'$ or $\gamma'$ into $\gamma$ by replacing some (possibly empty) subword of $\gamma$ with an equivalent subword using a relation $\gamma_A = \id_x$.
More generally, two parallel combinatorial paths are \defn{combinatorially homotopic} if they are related by a sequence of elementary combinatorial homotopies.

The quotient of the category $\mathcal{A}(X)$ by the relations $\alpha \alpha^{-1}=\alpha^{-1}\alpha=\id$ is the free groupoid $\mathcal{F}(X)$ generated by the $\alpha$ morphisms.
Let $\mathcal{C}(X)$ denote the further quotient of the groupoid $\mathcal{F}(X)$ by the relations $\gamma_A=\id_x$ for some choice of $x$, for each $2$-cell $A$ of $X$.
In other words, $\mathcal{C}(X)$ is the quotient of $\mathcal{A}(X)$ by the combinatorial homotopy equivalence relation.
Note that the definition of $\mathcal{C}(X)$ does not depend on the choice of $x$, for every $2$-cell $A$.
Indeed, if $x'\neq x \in A_0$ defines a relation $\gamma_A'=\id_{x'}$, we have $\gamma_A'=\delta \gamma_A \delta^{-1}$ in $\mathcal{F}(X)$, where $\delta$ is the morphism in  $\mathcal{A}(X)(x,x')$ induced by $\gamma_A$. 
Thus, a path~$\gamma$ can be rewritten into $\gamma'$ using $\gamma_A=\id_x$ if and only if it can be rewritten using $\gamma_A'=\id_{x'}$.

Let $\Pi(X)$ denote the \defn{fundamental groupoid }of $X$, that is the groupoid with objects the points of $X$ and morphisms the homotopy classes of paths between them.

\begin{thm}
\label{thm:top-coherence}
    Any two parallel combinatorial paths on $X$ are combinatorially homotopic if and only if every path component of $X$ is simply connected.
\end{thm}

\begin{proof}
    For $Y \subseteq X$, let us write $\Pi(X)Y$ for the full subcategory of the fundamental groupoid of $X$ spanned by $Y$.
    Then, we have an isomorphism of groupoids \[ \Pi(X)X_0 \cong \mathcal{C}(X) \ . \]
    To show this, one proceeds in three steps. 
    First, one shows that the fundamental groupoid $\Pi(X_1)X_0$ of the $1$-skeleton of $X$ is free on the homotopy classes of maps generated by the attaching maps of the $1$-cells, that is, free on the $\alpha$-morphisms \cite[9.1.5]{Brown2006}.
    Thus, one gets $\Pi(X_1)X_0 \cong \mathcal{F}(X)$. 
    Second, one shows that the fundamental groupoid $\Pi(X_2)X_0$ of the $2$-skeleton of $X$ is the free groupoid $\Pi(X_1)X_0$ modulo the relations $\gamma_A=1$, for $A$ a $2$-cell of $X$ \cite[9.1.6]{Brown2006}. 
    This is done through repeated application of the Seifert--Van Kampen theorem; one then has $\Pi(X_2)X_0 \cong \mathcal{C}(X)$.
    Third, one shows that the inclusion of $X_2$ in $X$ induces an isomorphism of fundamental groupoids $\Pi(X_2)X_0 \cong \Pi(X)X_0$ \cite[9.1.7]{Brown2006}, which concludes the proof of the isomorphism $\Pi(X)X_0 \cong \mathcal{C}(X)$.
    The theorem then follows, since every path component of $X$ is simply connected if and only if its fundamental groupoid $\Pi(X)$ is trivial, which holds if and only if its full subcategory $\Pi(X)X_0$ is trivial.  
\end{proof}

Note that any CW complex is locally path connected, and therefore is connected if and only if it is path connected.
Therefore, we could have replaced in the preceding theorem ``path component'' by ``connected component''.

Let us say that $X$ is \defn{combinatorially connected} if there is a combinatorial path between any two vertices of $X$. 
In the course of the preceding proof, we have in particular showed the following.
\begin{corollary}
    \label{cor:combinatorially-connected}
    A regular CW complex $X$ is combinatorially connected if and only it is connected.
\end{corollary}


\subsection{Coherence \`a la Morse}

Let $X\subset \R^n$ be a polyhedral complex. 
Let $\vec v \in \R^n$ be \defn{generic} on the edges of $X$, meaning that for any pair of vertices $x,y \in X$ belonging to the same edge of $X$, we have $\langle \vec v , x \rangle \neq \langle \vec v, y\rangle$.  
Such a generic vector $\vec v$ induces a natural orientation on the edges of $X$, directed from the source vertex where the functional $\langle \vec v, - \rangle$ is minimal to the target vertex where it is maximal. 

One of the basic, very useful facts about polyhedral complexes is that, for any face $F \subseteq X$ of $X$, there is a unique \defn{source} vertex $\so(F)$ such that all its adjacent edges $e \subseteq F$ are outgoing, and a unique \defn{sink} vertex $\sk(F)$ whose adjacent edges are all incoming, see \cite[Thm.~3.7]{Ziegler95}.
More generally  a vertex whose adjacent edges $e \subseteq X$ are all incoming is called a \defn{local sink}, and when $X$ has only one such vertex, we call it \defn{global sink} and denote it by $\sk(X)$.

Let $H:=\{y \in \R^n \ | \ \langle \vec v , y \rangle = 0\}$ be the linear hyperplane orthogonal to $\vec v$.  
For every vertex $x \in X$, choose $\varepsilon >0$ such that the interval between $\langle \vec v , x \rangle$ and $\langle \vec v , x \rangle + \varepsilon$ does not contain the image of any other vertex under the ``height" function $\langle \vec v, - \rangle$. 

\begin{definition}
    The \defn{outgoing link} $\oLk(x,X)$ of a vertex $x \in X$ is the intersection $\mathcal{F} \cap (H+x+\varepsilon \vec v)$ of the family of faces $\mathcal{F}(x,X):=\{ F \subseteq X \ | \ \so(F)=x \}$ with the affine hyperplane $H+x+\varepsilon \vec v$. 
\end{definition}

Recall from \cite[Sec.~2.1]{Ziegler95} that the \defn{vertex figure} $P/x$ of a polytope $P$ at a vertex $x$ is obtained by cutting $P$ by a hyperplane that cuts off the single vertex $x$. 
Such a cut establishes a bijection between the $(k-1)$-faces of $P/x$ and the $k$-faces of $P$ which contain $x$ \cite[Prop.~2.4]{Ziegler95}.

\begin{lemma}
    \label{l:vertex-figure}
    For any $k\geq 0$, there is a bijection between the $k$-faces of $\mathcal{F}(x,X)$ and the $(k-1)$-faces of $\oLk(x,X)$.
\end{lemma}
\begin{proof}
    Each maximal face of $\mathcal{F}(x,X)$ with respect to inclusion is a polytope $P$, for which the intersection $P \cap (H+x+\varepsilon \vec v)$ is the vertex figure $P/x$ of $P$ at $x$. 
    By \cite[Prop.~2.4]{Ziegler95}, there is a bijection between the $k$-faces of $P$ and the $(k-1)$-faces of $P/x$.
    Collecting these bijections for all maximal faces of $\mathcal{F}(x,X)$, and making the appropriate identifications, we get the desired global bijection.
\end{proof}

In this section we shall focus on polyhedral complexes whose outgoing links are connected. 
The following proposition gives the topological significance of this condition.

\begin{proposition}
    \label{lemma:outgoing-link}
    Let $X$ be a polyhedral complex.
    If there is a generic vector $\vec v \in \R^n$ such that the outgoing link of every vertex is connected, then every path component of $X$ is simply connected.
\end{proposition}

\begin{proof}
    Let $\vec v \in \R^n$ be generic with respect to $X$, and suppose that the outgoing link of every vertex is connected. 
    Since $\vec v$ is generic on edges, it defines a Morse function $\langle \vec v , -\rangle$ on $X$, in the sense of \cite[Def.~2.2]{bestvinaMorseTheoryFiniteness1997}.
    As in classical Morse theory, one can determine the homotopy type of $X$ by considering its successive level sets. 
    For $t \in \R$ denote by $X_t$ the closed subspace of $X$ containing points $x$ such that $\langle x, \vec v \rangle$ is at least $t$.
    Let $x$ be a vertex of $X$ of height $h=\langle x, \vec v \rangle$.
    Observe first that $X_{h+\epsilon}$, for some small $\epsilon>0$, is homotopy equivalent to $X_{h'}$ where $h' > h$ is the next greater height at which there is a vertex.
    That is, the homotopy type of $X$ can only change at vertices  \cite[Lem.~2.3]{bestvinaMorseTheoryFiniteness1997}.
    Then, one proves that $X_h$ is homotopy equivalent to the pushout of $X_{h+\epsilon}$ with the cone over the outgoing link of $x$ along the outgoing link of $x$  \cite[Lem.~2.5]{bestvinaMorseTheoryFiniteness1997}.
    By our assumption, the outgoing link of $x$ is connected, and thus the cone over it is simply connected. 
    Since the pushout of simply connected spaces over a connected space is always simply connected (this is an application of the Seifert--Van Kampen theorem), we obtain by induction that every path component of $X$ is simply connected \cite[Point (3) of Cor.~2.6]{bestvinaMorseTheoryFiniteness1997}.
\end{proof}

The converse of \cref{lemma:outgoing-link} is not true in general: many simply connected polyhedral complexes, as the one represented in \cref{fig:outgoingpoly}, have disconnected outgoing links, for many (sometimes for all) choices of generic orientation vectors. 

\begin{figure}[h!]
\centering
\resizebox{0.4\linewidth}{!}{
\begin{tikzpicture}
    \node[regular polygon,
    draw,
    regular polygon sides = 8, minimum size = 3cm] (p) at (0,0) {};
    \draw[-] (p.202.5)--(180:4)--(p.157.5);
    \draw[-] (p.157.5)--(135:4)--(p.112.5);
    \draw[-] (p.112.5)--(90:4)--(p.67.5);
    \draw[-] (p.67.5)--(45:4)--(p.22.5);
    \draw[-] (p.22.5)--(0:4)--(p.-22.5);
    \draw[-] (p.-22.5)--(-45:4)--(p.-67.5);
    \draw[-] (p.-67.5)--(-90:4)--(p.-112.5);
    \draw[-] (p.-112.5)--(-135:4)--(p.-157.5);
\end{tikzpicture}}
\caption{A simply connected polyhedral complex which admits disconnected outgoing links for every choice of generic vector.}
\label{fig:outgoingpoly}
\end{figure}
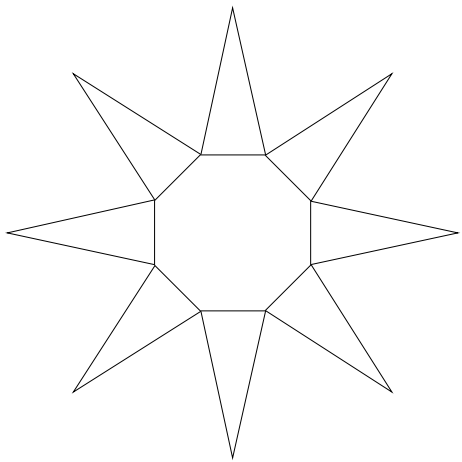

An important class of complexes which have connected outgoing links are polytopes, which will be our main object of study in the next sections.

\begin{proposition}
\label{prop:polytopes}
    Let $P \subset \R^n$ be a polytope, and let $\vec v \in \R^n$ be generic with respect to $P$. 
    Then, the outgoing link of every vertex of $P$ is connected.
\end{proposition}

\begin{proof}
    Define the linear hyperplane $H:=\{y \in \R^n \ | \ \langle \vec v, y \rangle = 0\}$, and consider the two half-spaces $H^{-}:=\{y \in \R^n \ | \ \langle \vec v, y \rangle < 0\}$ and $H^{+}:=\{y \in \R^n \ | \ \langle \vec v, y \rangle < 0\}$.
    Since $\vec v$ is not perpendicular to any edge of $P$, it defines a partition of the vertices of the vertex figure $P/x$ into two connected components: the vertices that lie in $H^{-}$, which correspond to incoming edges of $P$ at $x$, and the vertices that lie in $H^{+}$, which correspond to outgoing edges of $P$ at $x$.
    Thus, the outgoing link of $x$ is connected, and the proof is complete.
\end{proof}

From now on we shall suppose that the polyhedral complexes $X$ that we consider are endowed with a regular CW structure and provided with a generic vector $\vec v$.
Combining \cref{lemma:outgoing-link} with \cref{thm:top-coherence}, we have that any polyhedral complex $X$ whose outgoing links are connected satisfies the property that ``any two parallel combinatorial paths on $X$ are combinatorially homotopic''.
We shall now derive this same result by following an alternative, more combinatorial path (indeed!), getting close to the proof of \cite[Thm~3.1]{MacLane63}.

A combinatorial path $\gamma$ on $X$ is \defn{directed} if for any pair $(e, f)$ of consective edges in~$\gamma$, we have that $\sk(e)=\so(f)$.  
When no ambiguity arises, we will omit the adjective ``combinatorial" and say only ``directed path".

In the rest of this section we shall use the notion of combinatorial connectedness, which as we have seen in \cref{cor:combinatorially-connected} is equivalent to connectedness for the spaces we consider.
\begin{lemma}
    \label{l:unique-sink}
    Let $X$ be a polyhedral complex with generic vector $\vec v$ such that the outgoing link of every vertex is combinatorially connected. 
    Let $e,e'$ be two edges of $X$ such that $\so(e)=\so(e')$, and suppose that there are directed paths from $\sk(e)$ and $\sk(e')$ to local sinks~$s$ and~$s'$, respectively. 
    Then, we have $s=s'$.
\end{lemma}
\begin{proof}
    Define the \emph{height} $\frakh(x)$ of a vertex $x$ as the length of the longest directed path in $X$ starting at $x$.
    Since $\vec v$ is generic and $X_0$ is finite, this is well-defined.
    We proceed by induction on $\frakh(x)$.
    The statement holds vacuously for vertices $x$ such that $\frakh(x)=0$.
    Suppose that the assertion above holds for all vertices $x \in X$ such that $\frakh(x)=n$, and consider an $x$ with $\frakh(x)=n+1$.
    Since the outgoing link $\oLk(x,X)$ is combinatorially connected, there is a combinatorial path $\theta$ in $\oLk(x,X)$ between the vertices corresponding to $e$ and $e'$ (\cref{l:vertex-figure}).
    $\theta$ determines a sequence of edges $e_0 \eqdef e, e_1, \ldots, e_k, e' \defeq e_{k+1}$ of $X$ with $\so(e_i)=x$ for all $0 \leq i \leq k+1$.
    Moreover, each consecutive pair $e_i,e_{i+1}$ determines a $2$-face~$F_{i+1}$ of $X$.
    Now, choose for each $e_i$ with $1 \leq i \leq k$, a directed path of maximal length starting at $\sk(e_i)$ and passing through $\sk(F_{i})$.
    Each of these paths ends at a local sink $s_i$, including $s_0 \eqdef s$ and $s_{k+1} \eqdef s'$. 
    Since we have $\frakh(\sk(e_i))<\frakh(x)$ for all $0 \leq i \leq k+1$, we can apply induction to the two directed paths from $\sk(e_i)$ to $s_i$ and $s_{i+1}$, which gives~$s_i=s_{i+1}$.
    Therefore, we have $s=s_0=s_1=\cdots =s_k=s_{k+1}=s'$, as desired.
\end{proof}

Two parallel directed paths are said to be \defn{elementary combinatorially homotopic} if they are as undirected paths. 
They are \defn{combinatorially homotopic} if they are related by a sequence of elementary combinatorial homotopies between directed paths. 

The following \cref{l:oriented} and its consequence \cref{p:second-proof} expresses in topological terms the original proof technique used by Mac Lane in \cite[Thm~3.1]{MacLane63}.
Note that this \cref{l:oriented} involves first \emph{directed} paths, while \cref{p:second-proof} treats the general, undirected case.

\begin{proposition}
\label{l:oriented}
    Let $X$ be a polyhedral complex, and let $\vec v$ be generic on the edges of $X$. 
    Consider the following three properties:
    \begin{enumerate}
        \item[(i)] the outgoing link of every vertex is combinatorially connected,
        \item[(ii)] there is a global sink in every connected component,
        \item[(iii)] any two parallel \emph{directed} combinatorial paths on $X$ are combinatorially homotopic.
    \end{enumerate}
    Then, $X$ satisfies (i) if and only if it satisfies (ii) and (iii).
\end{proposition}

\begin{proof}
First, we prove that (i) implies (ii). 
    Suppose that there are two local sinks~$s_1$ and~$s_2$ in the same connected component of $X$.
    Consider a combinatorial path $\gamma$ between~$s_1$ and~$s_2$, whose existence is garanteed by \cref{cor:combinatorially-connected}.
    We proceed by induction on the number of \emph{peaks} in $\gamma$, that is the number of vertices $x$ which are the source $\so(e)=x=\so(e')$ of two edges $e,e'$ of $\gamma$.
    The path $\gamma$ has at least a peak, otherwise $s_1$ and $s_2$ would not be both local sinks. 
    If $\gamma$ has a unique peak, \cref{l:unique-sink} implies that $s_1=s_2$.
    Now suppose that for any $k \leq n$, if $\gamma$ has $k$ peaks, then we have $s_1=s_2$. 
    If $\gamma$ has $n+1$ peaks, consider the first peak $x=\so(e)=\so(e')$ of $\gamma$. 
    By \cref{l:unique-sink}, there is a directed path $\delta$ from $\sk(e')$ to $s_1$. 
    Replacing the initial section of $\gamma$ ending in $e’$ by $\delta$, we get a path with $n$ peaks, and by the induction hypothesis we get $s_1=s_2$, completing the proof.

   Second, we prove that (i) implies (iii).
    Let us assume that $X$ is connected, otherwise we apply the same reasoning to each connected component. 
    From the preceding paragraph, we know that $X$ has a global sink $\sk(X)$.
    Suppose that the outgoing link of every vertex is combinatorially connected. 
    Let $\gamma$ and $\gamma'$ be two parallel directed paths between two vertices $x$ and $y$. 
    We prove that they are combinatorially homotopic. 
    We proceed by induction on the maximal length $m$ of a directed path between $x$ and $y$ in $X$. 
    Without loss of generality, we can suppose that $y=\sk(X)$, since if $y\neq\sk(X)$ we can always find a directed path between $y$ and $\sk(X)$.
    The cases when $m=0$ and $m=1$ are trivial. 
    Suppose that the hypothesis holds up to $m=k-1, k\geq 2$, and consider two paths $\gamma$ and $\gamma'$ for which $m=k$. 
    Let $e$ and $e'$ denote the edges of $\gamma$ and $\gamma'$ that are adjacent to $x$. 
    We examine three cases.
    \begin{enumerate}
        \item If $e=e'$, we can apply the induction hypothesis to $\gamma \setminus e$ and $\gamma' \setminus e'$. 
        \item If $e \neq e'$ and both edges are on the same $2$-face $F$ of $X$, then using the induction hypothesis we have that $\gamma$ and $\gamma'$ are respectively combinatorially homotopic to the paths $\delta$ and $\delta'$ defined as follows: they go from $x=\so(F)$ to $\sk(F)$ by the unique path containing $e$ and $e'$, respectively, and then from $\sk(F)$ to $y$ along the same arbitrary directed path. 
        Since $\delta$ and $\delta'$ are combinatorially homotopic by definition, the conclusion follows from the transitivity of the combinatorial homotopy equivalence relation. 
        \item Suppose that $e\neq e'$, and that $e$ and $e'$ are \emph{not} on the same $2$-face of $X$. 
        Since the outgoing link of $x$ is combinatorially connected, there exists a combinatorial path $\theta$ between the vertices corresponding to $e$ and $e'$ in this link (\cref{l:vertex-figure}). 
        For every edge $e_i$ of $X$ in the path $\theta$, choose a directed path $\gamma_i$ in $X$ from $x$ to $y=\sk(X)$ going through $e_i$. 
        Now apply Point (2) above to every pair of parallel directed paths $(\gamma_i, \gamma_{i+1})$ with $e_i$ and $e_{i+1}$ consecutive in $\theta$, and conclude again by transitivity of the combinatorial homotopy equivalence relation. 
    \end{enumerate}

    Finally, we prove that (ii) and (iii) imply (i).
    Suppose that every pair of parallel directed combinatorial paths are combinatorially homotopic. 
    We show that for any vertex $x$, its outgoing link is combinatorially connected. 
    Indeed, take two edges $e,e'$ of $X$ with source $x$, and consider their extensions to directed paths $\gamma, \gamma'$ from $x$ to $\sk(X)$. 
    By hypothesis, these two paths are combinatorially homotopic, that is, there is a sequence of parallel directed paths from $\gamma$ to $\gamma'$. 
    The collection of first edges in each of these paths defines a combinatorial path between $e$ and $e'$ in the outgoing link of $x$. 
    Thus, this link is combinatorially connected. 
\end{proof}

\begin{thm}
\label{p:second-proof}
    Let $X$ be a polyhedral complex with generic vector $\vec v$ such that the outgoing link of every vertex is combinatorially connected.
    Then, any two parallel combinatorial paths on $X$ are combinatorially homotopic.
\end{thm}

\begin{proof} 
    Assume that $X$ is connected, otherwise apply the argument to each connected component.
    By \cref{l:oriented}, the polyhedral complex $X$ admits a global sink $\sk(X)$ and the conclusion holds for \emph{directed} paths.  
    Let us show that this implies the undirected version.
    Let $\gamma$ be an undirected combinatorial path on $X$ between $x$ and $y$.
    For every vertex $z$ along $\gamma$, one can choose a directed path $\delta_z$ from $z$ to $\sk(X)$. 
    We observe that for any edge $e: z \to z'$ of $\gamma$, the directed paths $\delta_z$ and $\delta_{z'}e$ are combinatorially homotopic by hypothesis. 
    Going from $x$ to $y$ inductively one edge at a time and using transitivity of the homotopy equivalence relation, one obtains that $\gamma$ is combinatorially homotopic to $\delta_y^{-1}\delta_x$. 
    Taking another combinatorial path $\gamma'$ parallel to $\gamma$, the same argument shows that $\gamma'$ is combinatorial homotopic to $\delta_y^{-1}\delta_x$.
    Thus $\gamma$ and $\gamma'$ are combinatorially homotopic, which completes the proof. 
\end{proof}

As \cref{lemma:outgoing-link} shows, the class of polyhedral complexes to which \cref{p:second-proof} applies is a strict subclass of simply connected complexes.
This implies that the converse of \cref{p:second-proof} does not hold, and thus that Mac Lane's original proof is far from reaching the full generality of \cref{thm:top-coherence}.
However, it will be sufficient for our purposes, since --as we have seen in \cref{prop:polytopes}-- it applies to any polytope.

\smallskip
Another feature of generically oriented polyhedral complexes is that their $1$-skeleton defines abstract rewriting systems which are terminating and confluent, as we now show.


\subsection{Rewriting systems}
\label{ss:abstract-rewriting}

We refer to \cite{baaderTermRewritingAll1998} for more details on rewriting systems. 
\begin{definition}
    An \defn{abstract rewriting system} is a set $A$ together with a binary relation~$\to$. 
\end{definition}
We denote by $\xrightarrow{*}$ the reflexive and transitive closure of ${\to}$. 
We say that $(A,\to)$ is \defn{locally confluent} (resp.\ \defn{confluent}) if for all $a,a_1,a_2 \in A$ such that $a_1 \leftarrow a \to a_2$ (resp.\ $a_1 \xleftarrow{*} a \xrightarrow{*} a_2$), there exists a term $b$ with $a_1 \xrightarrow{*} b \xleftarrow{*} a_2$.  
The diagram
\begin{center}
  \begin{tikzcd}
    & a \arrow[ld] \arrow[rd] &                     \\
a_1 \arrow[rd, "*"'] &                         & a_2 \arrow[ld, "*"] \\
    & b                       &                    
\end{tikzcd}
\end{center}
is called a \defn{local confluence diagram}.
A rewriting system is \defn{terminating} if every reduction sequence $a \to a_1 \to a_2 \to \cdots$ eventually must terminate.
An element $a \in A$ is \defn{reducible} if there exists an $a' \in A$ such that $a \to a'$; otherwise it is called \defn{irreducible} -- the rewriting synonymous of local sink!
We say that $b$ is a \defn{normal form} of $a$ if $a \xrightarrow{*} b$ and $b$ is irreducible.

    Given a polyhedral complex $X$ and a generic vector $\vec v$, one can consider the abstract rewriting system defined by $\vec v$ on the vertices of $X$.
\begin{definition}
    The \defn{vertices rewriting system} is the pair $(X_0,\to)$ made of the set of vertices $X_0$ of $X$, together with the following relation: we have $x \to y$ if $x$ and $y$ are vertices of the same edge and $\langle v, x \rangle < \langle v, y \rangle$.
\end{definition}
According to this definition, we have $x \xrightarrow{*} y$ if and only if there is a directed path from $x$ to $y$ in~$X_1$. 
The hypothesis of \cref{p:second-proof} imposes that the rewriting system $(X_0,\to)$ is terminating and confluent.
\begin{proposition}
    \label{prop:terminating-confluent}
    Let $X$ be a polyhedral complex and $\vec v$ be a generic vector. 
    If the outgoing link of every vertex is combinatorially connected, the rewriting system $(X_0,\to)$ is terminating and confluent. 
\end{proposition}
\begin{proof}
    Since $\vec v$ is generic, and thus strictly increasing along edges, it defines a partial order, and since the set $X_0$ is finite, the rewriting system $(X_0,\to)$ is terminating.
    By \cref{l:oriented}, there is a global sink in each connected component of $X$.
    Confluence then follows: given any pair of vertices $x,y$ in the same connected component, since $\vec v$ is generic there are directed paths $x \xrightarrow{*} s \xleftarrow{*} y$ to the global sink $s$ of this connected component. 
\end{proof}
\begin{corollary}
    The abstract rewriting system on the vertices of any oriented polytope~$P$ is terminating and confluent.
    Moreover, every pair of vertices admits a unique normal form~$\sk(P)$.
\end{corollary}
Recall that a polytope $P$ is \emph{simple} if each vertex of $P$ is incident to precisely $\dim P$ edges.
\begin{lemma}
    \label{l:simple-local-diagrams}
    If a polytope $P$ is simple, then there is a bijection between the local confluence diagrams of $(P_0,\to)$ and the oriented boundaries of the $2$-faces of~$P$.
\end{lemma}
\begin{proof}
    When $P$ is simple, the vertex figure $P/x$ of every vertex $x$ is a simplex \cite[Prop.~2.16]{Ziegler95}, with each edge in $P/x$ corresponding to a $2$-face of $P$ (\cref{l:vertex-figure}). 
    Thus every pair of edges $e,e'$ with source $x=\so(e)=\so(e')$ determines a $2$-face of~$P$. 
\end{proof}
Not much more can be said at this level of generality. 
For the specific familiy of operahedra that we will consider in the next section, the rewriting systems possess more structure (they are \emph{term} rewriting system) and are studied in a companion paper \cite{CLA24}.


%% file: sec/catoperads.tex

\section{Categorical coherence} 
\label{s:catoperads}

 
\subsection{Categorified non-symmetric operad}
\label{ss:def-catoperads}

Throughout this section we consider structures without units.
Unless otherwise stated, the adjective ``non-unital" will be implicitly assumed. 

\begin{definition} 
\label{def:catoperad}
A \defn{categorified non-symmetric operad} $\mathcal{P}$ is a collection $\left\{  \mathcal{P}(n)  \right\}_{n\in \mathbb{N}}$ of small categories equipped with bifunctors  
$$ \begin{array}{clll}
\circ_i&\colon& \mathcal{P}(n) \times
                    \mathcal{P}(k)
                    \longrightarrow \mathcal{P}(n+k-1) \ ,
                    & \text{for}\ 1 \leq i \leq n \ ,
\end{array}  $$
and for each $\kappa \in \mathcal{P}(m)$,  $\mu \in \mathcal{P}(n)$, $\nu \in \mathcal{P}(k)$, $1 \leq i \leq m$, $1 \leq j \leq n$ natural isomorphisms 
$$ \begin{array}{clll}
    \beta_{\kappa,\mu,\nu}&\colon& 
    (\kappa \circ_i \mu) \circ_{j+i-1} \nu  \overset{\cong}{\longrightarrow} \kappa \circ_i (\mu \circ_j \nu) \ , &  \\
    \theta_{\kappa,\nu,\mu}&\colon& 
    (\kappa \circ_i \nu) \circ_{j+k-1} \mu 
    \overset{\cong}{\longrightarrow} (\kappa \circ_j \mu) \circ_i \nu \ , & \text{when}\ i < j \ , 
\end{array}  $$
such that the following diagrams commute: the pentagonal \\
\begin{center}
\resizebox{\linewidth}{!}{
\begin{tikzpicture}[scale=2.5]
    \node (P1) at (0,1) {$((\kappa\circ\tau)\circ\mu)\circ\nu$};
    \node (P2) at (-0.866,0.5) {$(\kappa\circ(\tau\circ\mu)\circ\nu)$};
    \node (P3) at (-0.866,-0.5) {$\kappa\circ((\tau\circ\mu)\circ\nu)$};
    \node (P4) at (0,-1) {$\kappa\circ(\tau\circ(\mu\circ\nu))$};
    \node (P5) at (1,0) {$(\kappa\circ\tau)\circ(\mu\circ\nu)$} ;
    \draw[->] (P1)--(P2) node[midway,above left] {$\beta_{\kappa,\tau,\mu}\circ 1_\nu$};
    \draw[->] (P2)--(P3) node[midway,left] {$\beta_{\kappa,\tau\circ\mu,\nu}$};
    \draw[->] (P3)--(P4) node[midway,below left] {$1_\kappa \circ \beta_{\tau,\mu,\nu}$};
    \draw[->] (P1)--(P5) node[midway,above right] {$\beta_{\kappa\circ\tau,\mu,\nu}$};
    \draw[->] (P5)--(P4) node[midway,below right] {$\beta_{\kappa,\tau,\mu\circ\nu}$};
\end{tikzpicture} \quad 
\begin{tikzpicture}[scale=2.5]
    \node (P1) at (0,1) {$((\kappa\circ\tau)\circ\mu)\circ\nu$};
    \node (P2) at (-0.866,0.5) {$((\kappa\circ\tau)\circ\nu)\circ\mu$};
    \node (P3) at (-0.866,-0.5) {$((\kappa\circ\nu)\circ\tau)\circ\mu$};
    \node (P4) at (0,-1) {$(\kappa\circ\nu)\circ(\tau\circ\mu)$};
    \node (P5) at (1,0) {$(\kappa\circ(\tau\circ\mu)\circ\nu$} ;
    \draw[->] (P1)--(P2) node[midway,above left] {$\theta_{\kappa\circ\tau,\mu,\nu}$};
    \draw[->] (P2)--(P3) node[midway,left] {$\theta_{\kappa,\tau,\nu}\circ 1_\mu$};
    \draw[->] (P3)--(P4) node[midway,below left] {$\beta_{\kappa\circ\nu,\tau,\mu}$};
    \draw[->] (P1)--(P5) node[midway,above right] {$\beta_{\kappa,\tau,\mu}\circ 1_\nu$};
    \draw[->] (P5)--(P4) node[midway,below right] {$\theta_{\kappa,\tau\circ\mu,\nu}$};
\end{tikzpicture} } \\
\end{center}

\begin{center}
\resizebox{0.5\linewidth}{!}{
\begin{tikzpicture}[scale=2.5]
    \node (P1) at (0,1) {$((\kappa\circ\tau)\circ\mu)\circ\nu$};
    \node (P2) at (-0.866,0.5) {$((\kappa\circ\mu)\circ\tau)\circ\nu$};
    \node (P3) at (-0.866,-0.5) {$((\kappa\circ\mu)\circ\nu)\circ\tau$};
    \node (P4) at (0,-1) {$(\kappa\circ(\mu\circ\nu))\circ\tau$};
    \node (P5) at (1,0) {$(\kappa\circ\tau)\circ(\mu\circ\nu)$} ;
    \draw[->] (P1)--(P2) node[midway,above left] {$\theta_{\kappa,\tau,\mu}\circ 1_\nu$};
    \draw[->] (P2)--(P3) node[midway,left] {$\theta_{\kappa\circ\mu,\tau,\nu}$};
    \draw[->] (P3)--(P4) node[midway,below left] {$\beta_{\kappa,\mu,\nu}\circ 1_\tau$};
    \draw[->] (P1)--(P5) node[midway,above right] {$\beta_{\kappa\circ\tau,\mu,\nu}$};
    \draw[->] (P5)--(P4) node[midway,below right] {$\theta_{\kappa,\tau,\mu\circ\nu}$};
\end{tikzpicture}}
\end{center}

and hexagonal identities \\
\begin{center}
\resizebox{\linewidth}{!}{
\begin{tikzpicture}[scale=2.5]
    \node (P1) at (0,1) {$((\kappa\circ\tau)\circ\mu)\circ\nu$};
    \node (P2) at (-0.866,0.5) {$(\kappa\circ(\tau\circ\mu))\circ\nu$};
    \node (P3) at (-0.866,-0.5) {$\kappa\circ((\tau\circ\mu)\circ\nu)$};
    \node (P4) at (0,-1) {$\kappa\circ((\tau\circ\nu)\circ\mu)$};
    \node (P5) at (0.866,0.5) {$((\kappa\circ\tau)\circ\nu)\circ\mu$} ;
    \node (P6) at (0.866,-0.5) {$(\kappa\circ(\tau\circ\nu))\circ\mu$};
    \draw[->] (P1)--(P2) node[midway,above left] {$\beta_{\kappa,\tau,\mu}\circ 1_\nu$};
    \draw[->] (P2)--(P3) node[midway,left] {$\beta_{\kappa,\tau\circ\mu,\nu}$};
    \draw[->] (P3)--(P4) node[midway,below left] {$1_\kappa \circ \theta_{\tau,\mu,\nu}$};
    \draw[->] (P1)--(P5) node[midway,above right] {$\theta_{\kappa\circ\tau,\mu,\nu}$};
    \draw[->] (P5)--(P6) node[midway,right] {$\beta_{\kappa,\tau,\nu}\circ 1_\mu$};
    \draw[->] (P6)--(P4) node[midway,below right] {$\beta_{\kappa,\tau\circ\nu,\mu}$};
\end{tikzpicture} \quad \quad
\begin{tikzpicture}[scale=2.5]
    \node (P1) at (0,1) {$((\kappa\circ\tau)\circ\mu)\circ\nu$};
    \node (P2) at (-0.866,0.5) {$((\kappa\circ\mu)\circ\tau)\circ\nu$};
    \node (P3) at (-0.866,-0.5) {$((\kappa\circ\mu)\circ\nu)\circ\tau$};
    \node (P4) at (0,-1) {$((\kappa\circ\nu)\circ\mu)\circ\tau$};
    \node (P5) at (0.866,0.5) {$((\kappa\circ\tau)\circ\nu)\circ\mu$} ;
    \node (P6) at (0.866,-0.5) {$((\kappa\circ\nu)\circ\tau)\circ\mu$};
    \draw[->] (P1)--(P2) node[midway,above left] {$\theta_{\kappa,\tau,\mu}\circ 1_\nu$};
    \draw[->] (P2)--(P3) node[midway,left] {$\theta_{\kappa\circ\mu,\tau,\nu}$};
    \draw[->] (P3)--(P4) node[midway,below left] {$\theta_{\kappa,\mu,\nu}\circ 1_\tau$};
    \draw[->] (P1)--(P5) node[midway,above right] {$\theta_{\kappa\circ\tau,\mu,\nu}$};
    \draw[->] (P5)--(P6) node[midway,right] {$\theta_{\kappa,\tau,\nu}\circ 1_\mu$};
    \draw[->] (P6)--(P4) node[midway,below right] {$\theta_{\kappa\circ\nu,\tau,\mu}$};
\end{tikzpicture}  } \quad \ .
\end{center}
\end{definition}
The diagrams above hold for all instances of composable $\beta$ and $\theta$; these depend on the indices $i,j,k$, which are omitted for the sake of readability. 
Observe that a categorified non-symmetric operad concentrated in arity $1$ is a non-symmetric monoidal category.

As explained in \cref{prop:bijection-nestings} below, one can picture an object $\mu \in \mathcal{P}(n)$ as a planar tree with one vertex decorated by $\mu$, $n$ leaves and one root (a \defn{corolla}). 
The $\circ_i$ bifunctors then correspond to the operation of \defn{grafting} a corolla on top of another.
Iterated applications of the $\circ_i$ can be visualized as fully nested planar trees, with vertices decorated by objects of $\mathcal{P}$, see \cref{fig:treeandnesting}. 
A \defn{nesting} of a planar tree is a collection of subtrees (\defn{nests}) which are either included in one another or disjoint. 
A nesting is \defn{full} if its number of nests is maximal, equal to the number of internal edges of the tree \cite[Def.~2.2]{laplante-anfossiDiagonalOperahedra2022a}.

\begin{figure}[h!]
\centering
\resizebox{0.4\linewidth}{!}{
\begin{tikzpicture}[scale=1.2]
        \node (E)[circle,draw=black,minimum size=4mm,inner sep=0.1mm] at (-0,0) {\small $\kappa$};
        \node (F) [circle,draw=black,minimum size=4mm,inner sep=0.1mm] at (-1,1) {\small $\tau$};
        \node (A) [circle,draw=black,minimum size=4mm,inner sep=0.1mm] at (-1,2) {\small $\mu$};
        \node (q) [circle,draw=black,minimum size=4mm,inner sep=0.1mm] at (0,1) {\small $\nu$};
        \node (r) [circle,draw=black,minimum size=4mm,inner sep=0.1mm] at (1.65,1) {\small $\rho$};
        \node (x) [circle,draw=none,minimum size=4mm,inner sep=0.1mm] at (-0.4,0.62) {\color{Cyan} \small $1$};
        \node (y) [circle,draw=none,minimum size=4mm,inner sep=0.1mm] at (-0.875,1.6) {\color{Cerulean} \small $2$};
        \node (u) [circle,draw=none,minimum size=4mm,inner sep=0.1mm] at (0.15,0.65) {\color{NavyBlue} \small $3$};
        \node (v) [circle,draw=none,minimum size=4mm,inner sep=0.1mm] at (1.1,0.5) {\color{MidnightBlue} \small $4$};
        \draw[-] (0.8,0.8) -- (E)--(-1.65,0.8); 
        \draw[-] (-1.2,2.8) -- (A)--(-0.8,2.8); 
        \draw[-] (-0.2,1.8) -- (q)--(0.2,1.8);   
        \draw[-] (1.85,1.8) -- (r)--(1.45,1.8); 
        \draw[-] (E)--(0,-0.55); 
        \draw[-] (-1.4,1.8) -- (F)--(-0.6,1.8);   
        \draw[-] (E)--(F) node {};
        \draw[-] (E)--(q) node  {};
        \draw[-] (E)--(r) node {};
        \draw[-] (F)--(A) node {};
        \draw [Cyan,rounded corners,thick] (0.11,-0.32) -- (-0.14,-0.28) -- (-1.28,0.86) -- (-1.32,1.1) --  (-1.1,1.32) -- (-0.86,1.28) -- (0.28,0.14) -- (0.32,-0.11) -- cycle;
        \draw [Cerulean,rounded corners,thick] (0.14,-0.42) -- (-0.18,-0.36) -- (-1.2,0.6) -- (-1.4,0.9) -- (-1.3,2.1) -- (-1.15,2.3) -- (-0.85,2.3) -- (-0.7,2.1) --  (-0.7,1.3) -- (0.36,0.18) -- (0.42,-0.14) -- cycle;
        \draw [NavyBlue,rounded corners,thick] (0.18,-0.5) -- (-0.23,-0.45) -- (-1.3,0.6) -- (-1.5,0.9) --  (-1.35,2.3) --  (-1.15,2.4) --  (-0.85,2.4) --  (-0.7,2.3) --  (.25,1.2) -- (0.45,0.23) -- (0.5,-0.18) -- cycle;
        \draw [MidnightBlue,rounded corners,thick] (0.16,-0.6) -- (-0.25,-0.55) -- (-1.4,0.6) -- (-1.6,0.9) --  (-1.4,2.4) --  (-1.15,2.5) --  (-0.75,2.5) --  (1.8,1.2) --  (1.95,1) --  (1.8,0.8) -- cycle;
\end{tikzpicture} }
\caption{A fully nested planar tree.}
\label{fig:treeandnesting}
\end{figure}
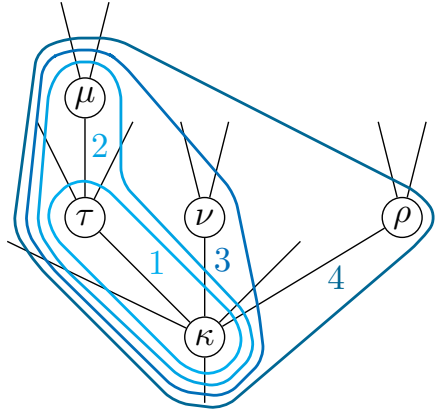 

The $\beta$ and $\theta$ arrows correspond to the sequential and parallel axioms of an operad, and relate the two possible ways of fully nesting a tree with $3$ vertices, see \cref{fig:betatheta}. 
Moreover, there is then one coherence diagram (pentagon or hexagon) for every planar tree with $4$ vertices, see \cref{fig:planar4trees}.

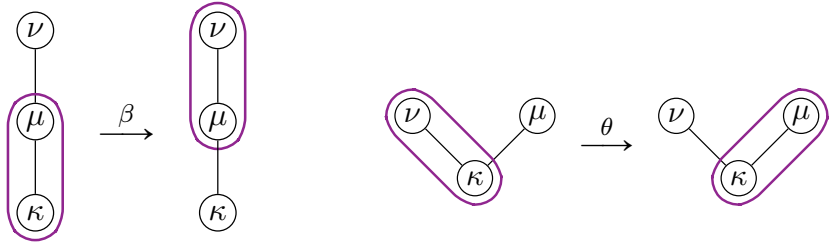
\begin{figure}[h!]
\begin{center}
\resizebox{0.25\linewidth}{!}{
\quad
\begin{tikzpicture}
    \node (t3)[circle,draw=black,minimum size=4mm,inner sep=0.1mm] at (0,0) {\small $\kappa$};
    \node (t2) [circle,draw=black,minimum size=4mm,inner sep=0.1mm] at (0,1) {\small $\mu$};
    \node (t1) [circle,draw=black,minimum size=4mm,inner sep=0.1mm] at (0,2) {\small $\nu$};  
    \draw[-] (t3)--(t2) node {};
    \draw[-] (t2)--(t1) node {};
    \draw [Plum,rounded corners,thick] (-0.15+2,0.7) -- (-0.3+2,0.9) -- (-0.3+2,2.1) -- (-0.15+2,2.3) -- (0.15+2,2.3) -- (0.3+2,2.1) -- (0.3+2,0.9) -- (0.15+2,0.7) -- cycle;

    \node (B) at (1,1) {$\overset{\beta}{\longrightarrow}$};

    \node (b3)[circle,draw=black,minimum size=4mm,inner sep=0.1mm] at (2,0) {\small $\kappa$};
    \node (b2) [circle,draw=black,minimum size=4mm,inner sep=0.1mm] at (2,1) {\small $\mu$};
    \node (b1) [circle,draw=black,minimum size=4mm,inner sep=0.1mm] at (2,2) {\small $\nu$};  
    \draw[-] (b3)--(b2) node {};
    \draw[-] (b2)--(b1) node {};
    \draw [Plum,rounded corners,thick] (1.85-2,-0.3) -- (1.7-2,-.1) -- (1.7-2,1.1) -- (1.85-2,1.3) -- (2.15-2,1.3) -- (2.3-2,1.1) -- (2.3-2,-.1) -- (2.15-2,-.3) -- cycle;
\end{tikzpicture}} \quad\quad \raisebox{2.75em}{}\quad\quad 
\resizebox{0.4\linewidth}{!}{
\raisebox{1em}{\begin{tikzpicture}
    \node (t3)[circle,draw=black,minimum size=4mm,inner sep=0.1mm] at (0,0) {\small $\kappa$};
    \node (t2) [circle,draw=black,minimum size=4mm,inner sep=0.1mm] at (0.71,0.71) {\small $\mu$};
    \node (t1) [circle,draw=black,minimum size=4mm,inner sep=0.1mm] at (-0.71,0.71) {\small $\nu$};  
    \draw[-] (t3)--(t1) node {};
    \draw[-] (t2)--(t3) node {};
    \draw [Plum,rounded corners,thick] (0.11,-0.32) -- (-0.14,-0.28) -- (-0.99,0.57) -- (-1.03,0.81) -- (-0.81,1.03) -- (-0.57,0.99)  -- (0.28,0.14) -- (0.32,-0.11) -- cycle;

    \node (B) at (1.5,0.5) {$\overset{\theta}{\longrightarrow}$};

    \node (b3)[circle,draw=black,minimum size=4mm,inner sep=0.1mm] at (3,0) {\small $\kappa$};
    \node (b2) [circle,draw=black,minimum size=4mm,inner sep=0.1mm] at (3+0.71,0.71) {\small $\mu$};
    \node (b1) [circle,draw=black,minimum size=4mm,inner sep=0.1mm] at (3-0.71,0.71) {\small $\nu$};  
    \draw[-] (b3)--(b1) node {};
    \draw[-] (b2)--(b3) node {};
    \draw [Plum,rounded corners,thick] (3-0.32,-0.11) -- (3-0.28,0.14) -- (3+0.57,0.99) -- (3+0.81,1.03) -- (3+1.03,0.81) -- (3+0.99,0.57) -- (3+0.14,-0.28) -- (3-0.11,-0.32) -- cycle;
\end{tikzpicture}}} \raisebox{2.75em}{} 
\end{center} 
\caption{The $\beta$ and $\theta$ isomorphisms defining a categorified non-symmetric operad.}
\label{fig:betatheta}
\end{figure}

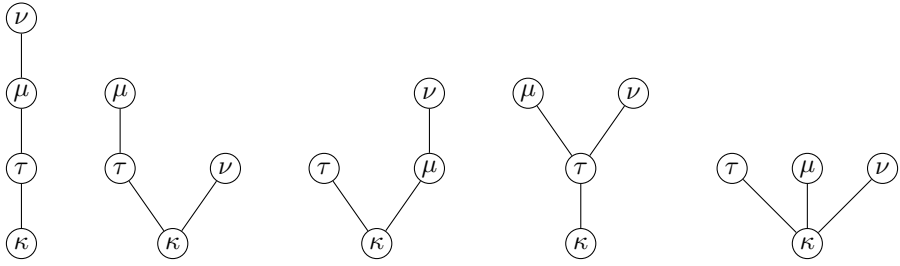
\begin{figure}[h!]
\centering 
\begin{tikzpicture}
    \node (t3)[circle,draw=black,minimum size=4mm,inner sep=0.1mm] at (0,0) {\small $\kappa$};
    \node (t2) [circle,draw=black,minimum size=4mm,inner sep=0.1mm] at (0,1) {\small $\tau$};
    \node (t1) [circle,draw=black,minimum size=4mm,inner sep=0.1mm] at (0,2) {\small $\mu$};  
    \node (t4) [circle,draw=black,minimum size=4mm,inner sep=0.1mm] at (0,3) {\small $\nu$};  
    \draw[-] (t4)--(t1) node {};
    \draw[-] (t3)--(t2) node {};
    \draw[-] (t2)--(t1) node {};
\end{tikzpicture}
\quad \quad
\begin{tikzpicture}
    \node (t3)[circle,draw=black,minimum size=4mm,inner sep=0.1mm] at (0,0) {\small $\kappa$};
    \node (t2) [circle,draw=black,minimum size=4mm,inner sep=0.1mm] at (-0.7,1) {\small $\tau$};
    \node (t1) [circle,draw=black,minimum size=4mm,inner sep=0.1mm] at (-0.7,2) {\small $\mu$};  
    \node (t4) [circle,draw=black,minimum size=4mm,inner sep=0.1mm] at (0.7,1) {\small $\nu$};  
    \draw[-] (t4)--(t3) node {};
    \draw[-] (t3)--(t2) node {};
    \draw[-] (t2)--(t1) node {};
\end{tikzpicture}
\quad \quad
\begin{tikzpicture}
    \node (t3)[circle,draw=black,minimum size=4mm,inner sep=0.1mm] at (0,0) {\small $\kappa$};
    \node (t2) [circle,draw=black,minimum size=4mm,inner sep=0.1mm] at (-0.7,1) {\small $\tau$};
    \node (t1) [circle,draw=black,minimum size=4mm,inner sep=0.1mm] at (0.7,1) {\small $\mu$};  
    \node (t4) [circle,draw=black,minimum size=4mm,inner sep=0.1mm] at (0.7,2) {\small $\nu$};  
    \draw[-] (t4)--(t1) node {};
    \draw[-] (t3)--(t1) node {};
    \draw[-] (t2)--(t3) node {};
\end{tikzpicture}
\quad \quad
\begin{tikzpicture}
    \node (t3)[circle,draw=black,minimum size=4mm,inner sep=0.1mm] at (0,0) {\small $\kappa$};
    \node (t2) [circle,draw=black,minimum size=4mm,inner sep=0.1mm] at (0,1) {\small $\tau$};
    \node (t1) [circle,draw=black,minimum size=4mm,inner sep=0.1mm] at (-0.7,2) {\small $\mu$};  
    \node (t4) [circle,draw=black,minimum size=4mm,inner sep=0.1mm] at (0.7,2) {\small $\nu$};  
    \draw[-] (t4)--(t2) node {};
    \draw[-] (t3)--(t2) node {};
    \draw[-] (t2)--(t1) node {};
\end{tikzpicture}
\quad \quad
\begin{tikzpicture}
    \node (t3)[circle,draw=black,minimum size=4mm,inner sep=0.1mm] at (0,0) {\small $\kappa$};
    \node (t2) [circle,draw=black,minimum size=4mm,inner sep=0.1mm] at (-1,1) {\small $\tau$};
    \node (t1) [circle,draw=black,minimum size=4mm,inner sep=0.1mm] at (0,1) {\small $\mu$};  
    \node (t4) [circle,draw=black,minimum size=4mm,inner sep=0.1mm] at (1,1) {\small $\nu$};  
    \draw[-] (t4)--(t3) node {};
    \draw[-] (t3)--(t2) node {};
    \draw[-] (t3)--(t1) node {};
\end{tikzpicture}
\caption{The five planar trees with four vertices, giving rise to the pentagonal (first three) and hexagonal (last two) identites.}
\label{fig:planar4trees}
\end{figure}

\begin{rem}
\label{rem:DPLA}
K. Do{\v s}en and Z. Petri{\'c} introduced in \cite[Sec.~12]{DP15} the notion of weak Cat-operad.
Despite looking different at first sight, the two notions are in fact equivalent.
The crucial observation is the following: the $\theta$-isomorphisms of Do{\v s}en--Petri{\'c} comprise both the isomorphisms~$\theta$ in \cref{def:catoperad} and their inverses $\theta^{-1}$.
Therefore, there are only two pentagonal coherence diagrams in the definition of a weak Cat-operad, the equations ($\beta$ $pent_e$) and ($\beta\theta 2_e$) of \cite[Section 9]{DP15}.
The set of diagrams of the form ($\beta$ $pent_e$) is the same as the set of diagrams which arises from the first pentagon in \cref{def:catoperad}, while the set of diagrams of the form ($\beta\theta 2_e$) is partitioned into the sets of diagrams which arise from the second and third pentagons in \cref{def:catoperad}.

We will give in \cref{thm:coherence-operahedra} two topological proofs of coherence for categorified non-symmetric operads.
A benefit of our presentation is that, adopting the oriented approach (see the second proof of \cref{thm:coherence-operahedra}), we get a proof of coherence where $\beta$ and $\theta$ are both treated as rewriting rules, in contrast with the proof in \cite{DP15}, which proceeds in two stages, much like in Mac Lane's proof of coherence for symmetric monoidal categories (see~\cref{MacLane-Kapranov-Simple}): first get rid of $\beta$ (rewriting), then deal with $\theta$. 
\end{rem}

\begin{definition}
    A \defn{strong morphism} of categorified non-symmetric operads $F: \mathcal{P} \to \mathcal{Q}$ is a collection of functors $F_n : \mathcal{P}(n) \to \mathcal{Q}(n)$ together with natural isomorphisms $$ \gamma_{\kappa,\mu}: F_{m-1+n}(\kappa \circ_i \mu) \overset{\cong}{\longrightarrow} F_m(\kappa) \circ_i F_n(\mu) $$ such that the following diagrams commute:
    \begin{center}
    \resizebox{\linewidth}{!}{
    \begin{tikzpicture}[scale=3]
        \node (P1) at (0,1) {$F((\kappa\circ\mu)\circ\nu)$};
        \node (P2) at (-0.866,0.5) {$F(\kappa\circ\mu)\circ F(\nu)$};
        \node (P3) at (-0.866,-0.5) {$(F(\kappa)\circ F(\mu))\circ F(\nu)$};
        \node (P4) at (0,-1) {$F(\kappa)\circ(F(\mu)\circ F(\nu))$};
        \node (P5) at (0.866,0.5) {$F(\kappa\circ(\mu\circ\nu))$} ;
        \node (P6) at (0.866,-0.5) {$F(\kappa)\circ F(\mu\circ\nu)$};
        \draw[->] (P1)--(P2) node[midway,above left] {$\gamma_{\kappa\circ\mu,\nu}$};
        \draw[->] (P2)--(P3) node[midway,left] {$\gamma_{\kappa\circ\mu}\circ 1_{F(\nu)}$};
        \draw[->] (P3)--(P4) node[midway,below left] {$\beta_{F(\kappa),F(\mu),F(\nu)}$};
        \draw[->] (P1)--(P5) node[midway,above right] {$F(\beta_{\kappa,\mu,\nu})$};
        \draw[->] (P5)--(P6) node[midway,right] {$\gamma_{\kappa,\mu\circ\nu}$};
        \draw[->] (P6)--(P4) node[midway,below right] {$1_{F(\kappa)}\circ\gamma_{\mu,\nu}$};
    \end{tikzpicture}  \quad \quad
    \begin{tikzpicture}[scale=3]
        \node (P1) at (0,1) {$F((\kappa\circ\nu)\circ\mu)$};
        \node (P2) at (-0.866,0.5) {$F(\kappa\circ\nu)\circ F(\mu)$};
        \node (P3) at (-0.866,-0.5) {$(F(\kappa)\circ F(\nu))\circ F(\mu)$};
        \node (P4) at (0,-1) {$(F(\kappa)\circ F(\mu))\circ F(\nu)$};
        \node (P5) at (0.866,0.5) {$F((\kappa\circ\mu)\circ\nu)$} ;
        \node (P6) at (0.866,-0.5) {$F(\kappa\circ\mu)\circ F(\nu)$};
        \draw[->] (P1)--(P2) node[midway,above left] {$\gamma_{\kappa\circ\nu,\mu}$};
        \draw[->] (P2)--(P3) node[midway,left] {$\gamma_{\kappa,\nu}\circ 1_{F(\mu)}$};
        \draw[->] (P3)--(P4) node[midway,below left] {$\theta_{F(\kappa),F(\nu),F(\mu)}$};
        \draw[->] (P1)--(P5) node[midway,above right] {$F(\theta_{\kappa,\nu,\mu})$};
        \draw[->] (P5)--(P6) node[midway,right] {$\gamma_{\kappa\circ\mu,\nu}$};
        \draw[->] (P6)--(P4) node[midway,below right] {$\gamma_{\kappa,\nu}\circ 1_{F(\nu)}$};
    \end{tikzpicture}  } \quad \ .
\end{center}
    It is said to be \defn{strict} if the natural isomorphisms are identities. 
\end{definition}

Once again, the diagrams above hold for all instances of $\beta$ and $\theta$ arrows, and we have omitted the $(i,j,k)$-indices for readability. 
Observe that a strong (resp. strict) morphism between categorified non-symmetric operads concentrated in arity $1$ is a strong (resp. strict) monoidal functor between non-symmetric monoidal categories. 


\subsection{Coherence for categorified non-symmetric operads}
\label{ss:coherence-catoperads}

We now aim at the coherence theorem for categorified non-symmetric operads.
In order to state the theorem, we construct the free non-symmetric categorified operad on a family of sets $S=\{S_n\}_{n \geq 1}$.
We define a family of categories~$\mathcal{S}=\{\mathcal{S}_n\}_{n \geq 1}$ whose objects are given by the following rules:
\begin{enumerate}
    \item if $a \in S_n$, then $a$ is an object of $\mathcal{S}_n$;
    \item if $t_1 \in \mathcal{S}_m$ and $t_2 \in \mathcal{S}_n$, then $t_1 \circ_i t_2$ is an object of $\mathcal{S}_{m-1+n}$, for any $1 \leq i \leq m$.
\end{enumerate}
If an object $t_1$ is in $\mathcal{S}_n$, we say that $t_1$ has \defn{arity} $n$. 
Now we define a set $M$ of \defn{basic morphisms} $\beta: (t_1 \circ_i t_2) \circ_{j+i-1} t_3 \leftrightarrow t_1 \circ_i (t_2 \circ_j t_3) : \beta^{-1}$ for every $t_1 \in \mathcal{S}_m, t_2 \in \mathcal{S}_n, t_3 \in \mathcal{S}_k$, $1 \leq i \leq m$ and $1 \leq j \leq n$, and $\theta: (t_1 \circ_i t_3) \circ_{j-1+k} t_2 \leftrightarrow (t_1 \circ_j t_2) \circ_i t_3 : \theta^{-1}$ whenever $i<j$.
We then define the \defn{generating morphisms} of the family $\mathcal{S}$ by the following rules:
\begin{enumerate}
    \item if $\phi \in M$, then $\phi$ is a generating morphism of $\mathcal{S}$; 
    \item if $\phi : t_1 \to t_2$ is a generating morphism in $\mathcal{S}$, and $t_3 \in \mathcal{S}$, then $\phi \circ_i \id : t_1 \circ_i t_3 \to t_2 \circ_i t_3$ and $\id \circ_j \phi : t_3 \circ_j t_1 \to t_3 \circ_j t_2$ are generating morphisms, for any $i$ (resp. $j$) between~$1$ and the arity of $t_1$ (resp. $t_3$).
\end{enumerate}
Note that by construction, for every morphism $\phi : t_1 \to t_2$, the objects $t_1$ and $t_2$ have the same arity, and we say that $\phi$ has this \defn{arity}. 
We then define $\mathcal{S}_n$ as the free category over all generating morphisms of arity $n$. 
This finishes the construction of our family $\mathcal{S}$ of categories.

\begin{definition}
    We denote by $\mathcal{F}(S)$ the quotient of the family of categories $\mathcal{S}$ by localization (inverting the $\beta$ and $\theta$ morphisms), the axioms of bifunctors for the $\circ_i$, the naturality conditions for $\beta$ and $\theta$, and the coherence diagrams (pentagons and hexagons) defining a categorified non-symmetric operad. 
\end{definition}

We obtain that $\mathcal{F}(S)$ is the free categorified non-symmetric operad on $S$. 
That is, for any categorified non-symmetric operad $\mathcal{P}$, and for any family of functions $\rho_n : S_n \to \obj(\mathcal{P}(n))$, there is a unique \emph{strict} morphism of non-symmetric categorified operads $\mathcal{F}(S) \to \mathcal{P}$ which extends $\rho=\{\rho_n\}_{n\geq 1}$. 
By precomposing it with the quotient map $\mathcal{S} \to \mathcal{F}(S)$, we get a levelwise functor $\extu{-}:\mathcal{S} \to \mathcal{P}$.

\begin{thm}[Coherence theorem]
\label{thm:coherence-operahedra}
    For any categorified non-symmetric operad $\mathcal{P}$, for any family of functions $\rho : S \to \obj(\mathcal{P})$, and for any two parallel morphisms $\phi_1,\phi_2: t_1 \to t_2$ in~$\mathcal{S}$, we have $\extu{\phi_1}=\extu{\phi_2}$.
\end{thm}

In order to prove this \cref{thm:coherence-operahedra}, we need to first recall the construction of the \defn{operahedra}, a family of polytopes whose faces are in bijection with the set of all nestings of a planar tree.
We refer to \cite[Sec.~2]{laplante-anfossiDiagonalOperahedra2022a} for details, see also \cite[Sec.~13]{DP15} and \cite{curienSyntacticAspectsHypergraph2019a}.
Given a planar tree $t$ with $n$ internal edges, and a full nesting $\mathcal{N}$ of $t$, one associates a point~$M(t,\mathcal{N}) \in \R^n$ via a simple algorithm which is due to J.-L. Loday \cite[Sec.~2.2]{laplante-anfossiDiagonalOperahedra2022a}.  
The \defn{operahedron}~$P_t \subset \R^n$ is the convex hull of the points $M(t,\mathcal{N})$, for all maximal nestings $\mathcal{N}$ of $t$.
It has dimension $n-1$. 
One then shows that the poset of nestings of~$t$, ordered by reverse inclusion, is isomorphic to the poset of faces of~$P_t$ \cite[Prop.~2.15]{laplante-anfossiDiagonalOperahedra2022a}. 
The dimension of a face is given by $n$ minus the number of nests in the corresponding nesting of $t$.

Reading a planar tree $t$ from the leaves to the root defines a family of \defn{incoming edges} and one \defn{outgoing edge} at each vertex of $t$.
Given the family of sets $S$ and a planar tree $t$, we say that a decoration of the vertices of $t$ by elements of $S$ is \defn{admissible} if at every vertex the number of incoming edges is equal to the arity of the element of $S$ decorating it. 
Now, let us consider the collection $\mathcal{O}(S)$ of polytopes with one copy of the operahedron~$P_t$ for each admissible decoration of the planar tree $t$ by elements of $S$. 
\begin{samepage}
    \begin{proposition}
        \label{prop:bijection-nestings}
            There are bijections between
            \begin{enumerate}
                \item objects of $\calS$ and vertices of the operahedra in $\mathcal{O}(S)$,
                \item generating morphisms of $\calS$ and edges of the operahedra in $\mathcal{O}(S)$,
                \item bifunctoriality, naturality and coherence diagrams and $2$-faces of the operahedra in~$\mathcal{O}(S)$.
            \end{enumerate}
        \end{proposition}
\end{samepage}
\begin{proof}
To each element $a$ of $S_n$, we associate a planar corolla with $n$ leaves and vertex decorated by $a$. 
Then, we identify $a \circ_i b$ with the planar tree obtained from grafting the corolla decorated by $b$ at the $i$th leaf of the corolla decorated by $a$.
Continuing in this fashion, and remembering the order in which we graft the corollas, we obtain all possible fully nested planar trees with vertices decorated by elements of $S$ (\cref{fig:treeandnesting}).
A generating morphism $f$ in $\mathcal{S}$ is an application of one of the associativity rules~$\beta$ or~$\theta$ to a fully nested tree $t$, moving only one nest (\cref{fig:betatheta}). 
If $t$ has $n$ internal edges, forgetting the nest that has been moved gives a nesting of $t$ with $n-1$ nests.
We associate to $f$ the edge of the operahedron $P_t$ in $\mathcal{O}(S)$ labeled by this nesting, see \cite[Def.~2.8 \& Prop.~3.11]{laplante-anfossiDiagonalOperahedra2022a}.
It remains to consider all the possible diagrams one can obtain by applying two generating morphisms to a given fully nested tree $t$ with $n$ internal edges. 
These arise from moving two different nests in the same fully nested tree. 
Starting by moving one or the other of these $2$ nests, one faces two types of situations:
\begin{itemize}
    \item[(A1)] If the two nests are disjoint, one obtains a bifunctoriality square,
    \item[(A2)] If the two nests are nested, but do not share the same root, one obtains a naturality square,
    \item[(B)] If the two nests are nested and share the same root, one obtains either a pentagon or a hexagon as in \cref{def:catoperad}.
\end{itemize}
To such a diagram, we associate the $2$-face of the operahedron~$P_t$ in $\mathcal{O}(S)$ corresponding to the nesting of $t$ obtained by forgetting the two nests that have been moved along the edges.
We refer to \cite[Sec.~2]{CLA24} for a more detailed analysis of the $2$-faces.  
\end{proof}
\begin{rem}
    The fact that every possible choice of initial moves gives rise to a $2$-face amounts to the fact that the operahedron $P_t$ is a \emph{simple} polytope \cite[Sec.~9]{DP-HP}.
    As \cref{l:simple-local-diagrams} shows, this property garantees the correspondence between geometric and rewriting-theoretic proofs of coherence, see \cite{CLA24} for more details on the latter.
\end{rem}
The conceptual origin of the bijections of \cref{prop:bijection-nestings} is the fact that the combinatorics of the faces of the operahedra correspond exactly to the monad of trees \cite[Sec.~5.6.1]{LodayVallette12}.
Or, said differently, it lies in the fact that the operahedra encode (via the cellular chains functor) the minimal resolution of the colored symmetric operad whose algebras are non-unital non-symmetric operads, see \cite{VanDerLaan03} and \cite[Sec.~4.1]{laplante-anfossiDiagonalOperahedra2022a}.

We are now ready to prove \cref{thm:coherence-operahedra}, using either our non-oriented or oriented topological coherence results for polytopes.
\begin{proof}[Proof of {\cref{thm:coherence-operahedra}}]
From Point (2) in \cref{prop:bijection-nestings}, we have that the morphisms of~$\mathcal{S}$ are in bijection with combinatorial paths on the operahedra of $\mathcal{O}(S)$.
Two parallel morphisms in~$\mathcal{S}$ thus define two parallel combinatorial paths on some operahedron $P_t$ in $\mathcal{O}(S)$. 
Since an operahedron $P_t$ is simply connected, \cref{thm:top-coherence} implies that these two combinatorial paths are combinatorially homotopic. 
By Point (3) in \cref{prop:bijection-nestings} the $2$-faces of the operahedra are exactly either a bifunctoriality or naturality square, a pentagon or a hexagon (witnessing a coherence condition) as in \cref{def:catoperad}.
Therefore, two parallel morphisms $\phi_1,\phi_2$ in $\mathcal{S}$ are equal in the quotient $\mathcal{F}(S)$, and thus their images $\extu{\phi_1},\extu{\phi_2}$ are also equal in $\mathcal{P}$.
\end{proof}

\begin{proof}[Second proof of {\cref{thm:coherence-operahedra}}]
    Alternatively, since the operahedra are polytopes, one can use \cref{prop:polytopes} and \cref{p:second-proof}. 
    As shown in \cite[Prop.~3.11]{laplante-anfossiDiagonalOperahedra2022a}, choosing a generic vector $\vec v$ which has strictly decreasing coordinates gives the orientations of the diagrams given in \cref{def:catoperad} on the $2$-faces.
    One then obtains a topological proof of coherence which follows closely the original proof of Mac Lane \cite[Thm.~3.1]{MacLane63}, suitably generalized to categorified operads. 
\end{proof}
   
Following \cref{rem:DPLA}, we have that \cref{thm:coherence-operahedra} gives an alternative, more economical proof of coherence for weak Cat-operads \cite[Prop.~14.2]{DP15}.
Incidentally, it gives an alternative input to the proof of coherence for cyclic symmetric categorified operads \cite{curienCategorifiedCyclicOperads2020}.

Restricting the theorem above to non-symmetric operads concentrated in arity $1$, the category $\mathcal{F}(S)$ becomes the free non-symmetric monoidal category on $S$, and we get the following corollary. 

\begin{corollary}[Mac Lane coherence theorem for non-symmetric monoidal categories]
\label{cor:MacLane}
    For any non-symmetric monoidal category $\mathcal{C}$, for any function $\rho : S \to \obj(\mathcal{C})$, and for any two parallel morphisms $\phi_1,\phi_2: t_1 \to t_2$ in $\mathcal{S}$, we have $\extu{\phi_1}=\extu{\phi_2}$.
\end{corollary}

\begin{rem}
    As mentioned at the end of \cref{ss:abstract-rewriting}, the rewriting systems obtained on the vertices of the operahedra by choosing a generic vector with strictly decreasing coordinates are in fact \emph{term} rewriting systems.
    The faces of type (B) in Point (3) of \cref{prop:bijection-nestings} (the coherence conditions) correspond precisely to the \emph{critical pairs} of these rewriting systems, see \cite[Sec.~3.4]{CLA24}.
    Moreover, the associated posets on fully nested planar trees have recently been shown to be lattices \cite{DefanSack24}.
\end{rem}

%% file: sec/monoidal.tex


\subsection{Symmetric monoidal categories} 

We now formulate and prove Mac Lane's coherence theorem for \emph{symmetric} monoidal categories in the same style as above.  Recall that in a symmetric monoidal category $\mathcal{C}$, in addition to the natural isomorphisms $\beta$, with components $\beta_{\kappa,\mu,\nu}:(\kappa\otimes\mu)\otimes\nu\to
\kappa\otimes(\mu\otimes\nu)$, there are involutive natural transformations $\tau$, with components $\tau_{\mu,\nu}:\mu\otimes\nu\to\nu\otimes\mu$.
Here, we use $\kappa,\mu,\nu,\ldots$ to range over the objects of the category, consistently with the notation used in~\cref{ss:def-catoperads,ss:coherence-catoperads}.
In addition to the pentagons, obtained from the first pentagon in~\cref{def:catoperad} by replacing $\circ$ with $\otimes$, there are hexagons 
\begin{center}
\resizebox{0.4\linewidth}{!}{
    \begin{tikzpicture}[scale=2.5]
    \node (P1) at (0,1) {$(\kappa\otimes \mu) \otimes \nu$};
    \node (P2) at (-0.866,0.5) {$\kappa\otimes (\mu \otimes \nu)$};
    \node (P3) at (-0.866,-0.5) {$(\mu \otimes \nu) \otimes \kappa$};
    \node (P4) at (0,-1) {$\mu \otimes (\nu \otimes \kappa)$};
    \node (P5) at (0.866,0.5) {$(\mu \otimes \kappa) \otimes \nu$} ;
    \node (P6) at (0.866,-0.5) {$\mu \otimes (\kappa \otimes \nu)$};
    \draw[->] (P1)--(P2) node[midway,above left] {$\beta$};
    \draw[->] (P2)--(P3) node[midway,left] {$\tau$};
    \draw[->] (P3)--(P4) node[midway,below left] {$\beta$};
    \draw[->] (P1)--(P5) node[midway,above right] {$\tau\otimes 1$};
    \draw[->] (P5)--(P6) node[midway,right] {$\beta$};
    \draw[->] (P6)--(P4) node[midway,below right] {$1\otimes\tau$};
\end{tikzpicture}}
\end{center}
for all objects $\kappa,\mu,\nu$ in $\mathcal{C}$.

In order to state the coherence theorem, we construct a free category on a set $S$ of \defn{generating objects}. 
We define a small category $\mathcal{S}^{\mathrm{ML}}$ whose set of objects $${\cal T}_S=\bigcup\setc{{\cal T}_A}{A \;\mbox{is a non-empty finite subset of}\;S}$$ is defined as follows:
\begin{enumerate}
    \item if $a \in S$, then $a\in{\cal T}_{\set{a}}$;
    \item if $t_1 \in {\cal T}_A$ and $t_2 \in {\cal T}_B$, and if $A\cap B=\emptyset$,  then $t_1 \otimes t_2\in{\cal T}_{A\cup B}$.
\end{enumerate}
We can see the objects of $\mathcal{S}^{\mathrm{ML}}$ as fully parenthesized words over $S$ where letters are not repeated.
We then define a set $M^{\mathrm{ML}}$ of \defn{basic morphisms} $\beta: (t_1 \otimes t_2) \otimes t_3 \leftrightarrow t_1 \otimes (t_2\otimes t_3) : \beta^{-1}$  and $\tau: t_1\otimes t_2 \leftrightarrow t_2 \otimes t_1$,  for every $t_1,t_2,t_3  \in {\cal T}_S$.
We then define the \defn{generating morphisms} of $\mathcal{S}^{\mathrm{ML}}$ by the following rules:
\begin{enumerate}
    \item if $\phi \in M^{\mathrm{ML}}$, then $\phi$ is a generating morphism; 
    \item if $\phi : t_1 \to t_2$ is a generating morphism and $t_3 \in {\cal T}_S$, then $\phi \otimes \id : t_1 \otimes t_3 \to t_2 \otimes t_3$ and $\id \otimes \phi : t_3 \otimes t_1 \to t_3 \otimes t_2$ are generating morphisms.
\end{enumerate}
We then define $\mathcal{S}^{\mathrm{ML}}$ as the free category over all generating morphisms. 
This finishes the construction of the category $\mathcal{S}^{\mathrm{ML}}$.
\begin{definition}
    We denote by $\mathcal{F}(S)$ the quotient of $\mathcal{S}^{\mathrm{ML}}$ by localization (inverting the $\beta$ morphisms), by the axioms $\tau_{t_1,t_2}\circ \tau_{t_2,t_1}=1$, by the axioms of bifunctors, by the naturality conditions for $\beta$ and $\tau$, and by the coherence conditions of symmetric monoidal categories.
\end{definition}
By freeness, we have that for any symmetric monoidal category $\mathcal{C}$, and for any function $\rho : S \to \obj(\mathcal{C})$, there is a unique functor  $\extd{-}{\mathrm{ML}}:\mathcal{S}^{\mathrm{ML}} \to \mathcal{C}$ which extends $\rho$ and sends the formal basic morphisms to the actual canonical morphisms of $\mathcal{C}$.
This functor factorizes  through the quotient map $[-]^{\mathrm{ML}} :\mathcal{S}^{\mathrm{ML}} \to \mathcal{F}(S)$.

It turns out that Kapranov's topological proof is not based on the above presentation of~$\mathcal{F}(S)$, but on another presentation of this category, that is made explicit in~\cite[Sec.~2]{baralicSimplePermutoassociahedron2019}. 
Let us recall this presentation. 
We define another category $\mathcal{S}^{\mathrm{K}}$ as follows. Its objects are the same as those of $\mathcal{S}^{\mathrm{ML}}$. We define a set $M^{\mathrm{K}}$ of \defn{basic morphisms} $\beta: (t_1 \otimes t_2) \otimes t_3 \leftrightarrow t_1 \otimes (t_2\otimes t_3) : \beta^{-1}$ for every $t_1,t_2,t_3  \in {\cal T}_S$, and $\tau: a\otimes b \leftrightarrow b \otimes a$ for every $a,b \in S$, i.e., we \emph{limit $\tau$ to generating objects}.
\defn{Generating morphisms} are defined in the same way as for $\mathcal{S}^{\mathrm{ML}}$.
We note that by construction $\mathcal{S}^{\mathrm{K}}$ is a wide subcategory of $\mathcal{S}^{\mathrm{ML}}$.
\begin{definition}
    \label{def:free-Kap}
    We denote  by $\mathcal{F}(S)^{\mathrm{K}}$ the quotient of $\mathcal{S}^{\mathrm{K}}$ by localization (inverting the $\beta$ morphisms), by the axioms $\tau_{a,b}\circ \tau_{b,a}=1$, by the axioms of bifunctors, by the naturality conditions for $\beta$, by the coherence conditions of  monoidal categories, and by the axioms in dodecagonal form given by the solid arrows in~\cref{fig:dodecagon} (left), for $a,b,c$ ranging over $S$ only.
\end{definition}

\begin{figure}[h!]
	\centerline{\includegraphics[scale=.5]{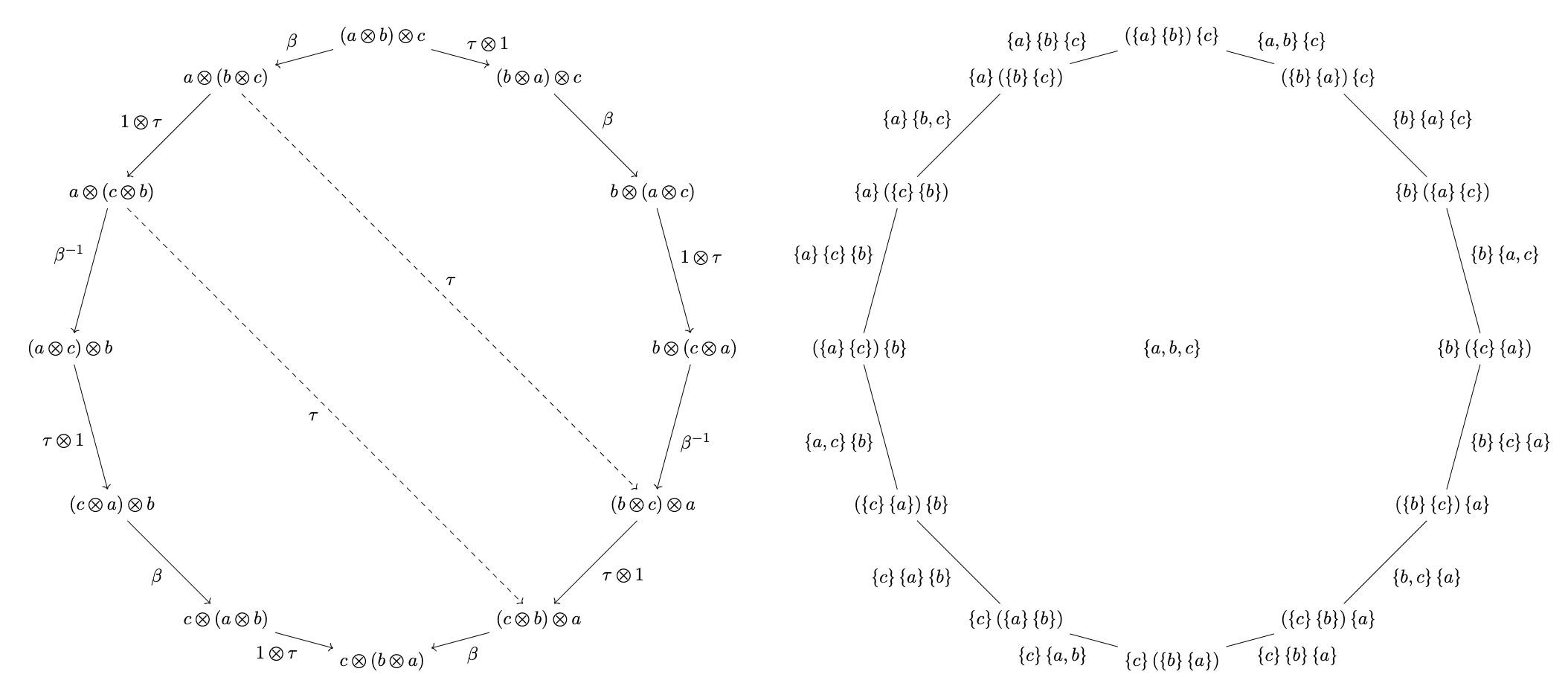}}
	\caption{Kapranov dodecagons.}
    \label{fig:dodecagon}
\end{figure}

We pause here to reflect on the difference between the two presentations. In the second one, we have less generators, and we have lost hexagons. For an intuition, here is how Mac Lane himself motivated his hexagonal axioms (verbatim, just changing the notation to fit with ours) in~\cite{MacLane63}:
\begin{quote}
    \label{q:ML}
The instance $\tau_{\kappa\otimes \mu,\nu}$ interchanges the block $\kappa\mu$ with the single letter $\nu$; the hexagon condition states that this interchange may be replaced by two instances of $\tau$ which interchange single letters with $\nu$. Repeated such replacement using instances of the hexagon shows that any interchange of successive blocks may be replaced by interchanges of successive letters.
\end{quote}
In other words, hexagons are now taken as definitions rather than axioms. 
But how do we guarantee that the  general $\tau$ morphisms defined in this way define a natural transformation? 
This is what the dodecagons are for.

Let $\mathcal{C}$  be a symmetric monoidal category. 
By freeness again, any function $\rho : S \to \obj(\mathcal{C})$ extends uniquely to a functor $\extd{-}{\mathrm{K}}:\mathcal{S}^{\mathrm{K}} \to \mathcal{C}$.
This functor is the restriction of $\extd{-}{\mathrm{ML}}$ to $\mathcal{S}^{\mathrm{K}}$, and factorizes through the quotient functor $[-]^{\mathrm{K}}:\mathcal{S}^{\mathrm{K}}\to  \mathcal{F}(S)^{\mathrm{K}}$.
\begin{thm}[Kapranov coherence theorem for symmetric monoidal categories]
\label{thm:coherence-Kapranov}
 For any two parallel morphisms $\phi_1,\phi_2: t_1 \to t_2$ in~$\mathcal{S}^{\mathrm{K}}$, we have~$[\phi_1]^{\mathrm{K}}=[\phi_2]^{\mathrm{K}}$.
\end{thm}

In order to prove this ``Kapranov style'' coherence, we need to first recall the construction of the \defn{permuto-associahedra}, a family of polytopes whose faces are in bijection with parenthesized ordered partitions of a finite set.
We refer to \cite[Sec.~9.3]{Ziegler95} for details, see also \cite{kapranov1993} and \cite{reinerCoxeterassociahedra1994}.
Given a finite set $A$ of cardinal $n$ and a parenthesized permutation $\sigma$ of its elements, one associates a section $\gamma^\sigma$ of the projection from the $n$-cube to the cyclic polygon with $n+1$ vertices \cite[Ex.~9.14]{Ziegler95}, whose integral over the base gives a point~$M(\sigma)$ in $\R^n$. 
The \defn{permuto-associahedron} $P_A$ is the convex hull of the points $M(\sigma)$, for all parenthesized permutations $\sigma$ of the elements of $A$. 
It has dimension $n-1$. 
One then shows that the poset of parenthesized ordered partitions of $A$, ordered according to the rules below, is isomorphic to the poset of faces of $P_A$ \cite[Thm.~9.15]{Ziegler95}.

Parenthesized ordered partitions of $A$ can be drawn as planar trees whose leaves are decorated with the parts of a partition of $A$.
The subface relation $\prec$ is defined by two clauses: one can contract an edge of the tree, or remove a node all of whose  incoming edges are leaves and decorate its outcoming edge -- now a leaf -- with the union of the decorations of those incoming edges.
The maximal face is $A$.
For example, with $A=\set{a_1,\ldots,a_7}$, the following is a face:
$$\begin{array}{l}
(\set{a_1} \set{a_4} \set{a_2,a_6}) \set{a_3,a_5,a_7}
\end{array}$$
which is covered by the following two elements.
$$\begin{array}{lll}
(\set{a_1} \set{a_4} \set{a_2,a_6}) \set{a_3,a_5,a_7} & \prec & \set{a_1} \set{a_4} \set{a_2,a_6} \set{a_3,a_5,a_7} \\
(\set{a_1} \set{a_4} \set{a_2,a_6}) \set{a_3,a_5,a_7} & \prec &  \set{a_1,a_2,a_4,a_6} \set{a_3,a_5,a_7}.
\end{array}$$

Given the set $S$, let us consider the collection $\mathcal{P}(S)$ of polytopes with one copy of the permuto-associahedron $P_A$ for each finite subset $A \subset S$.
\begin{proposition}
\label{bijections-Kapranov}
    There are bijections between
\begin{enumerate}
    \item objects of $\mathcal{S}^{\mathrm{K}}$ and vertices of the permuto-associahedra in $\mathcal{P}(S)$,
    \item generating morphisms of $\mathcal{S}^{\mathrm{K}}$ and edges of the permuto-associahedra in $\mathcal{P}(S)$,
    \item bifunctoriality, naturality and coherence diagrams and $2$-faces of the permuto-associahedra in $\mathcal{P}(S)$.
\end{enumerate}
\end{proposition}
\begin{proof}
    The 0-dimensional faces of $P_A$ are fully parenthesized words whose letters are singletons, and are in obvious bijective correspondence with the elements of ${\cal T}_A$. 
The 1-dimensional faces are 
\begin{itemize}
\item
either fully parenthesized words whose letters are singletons but for one letter which is a two-element set $\set{a_i,a_j}$ and feature an application of the basic morphism $\tau_{a_i,a_j}$, 
\item or an ``almost'' fully parenthesized word of singletons, with just one parenthesis removed, yielding a subword $(\set{a_i} \set{a_j} \set{a_k})$, featuring an application of the basic morphism $\beta_{a_i,a_j,a_k}$ or $\beta_{a_i,a_j,a_k}^{-1}$
-- the orientation of the edge being ``decided'' by the shape of its end vertices.
\end{itemize}
Finally, the 2-dimensional faces can be analyzed much in the same way as in~\cref{prop:bijection-nestings}, and seen to correspond to bifunctoriality, naturality of $\beta$, and to the pentagons and dodecagons. 
We have pictured the poset view of the 
latter in~\cref{fig:dodecagon} (right). 
The reader can also convince himself on this figure how the orientation of the $\beta$ arrows on the left can be reconstructed from the non-oriented dodecagon on the right.
\end{proof}
\begin{proof}[Proof of \cref{thm:coherence-Kapranov}]
The proof is similar to the proof of \cref{thm:coherence-operahedra}, using either the Van Kampen (\cref{thm:top-coherence}) or the Morse (\cref{prop:polytopes,p:second-proof}) technique.
\end{proof}
\begin{rem}
\label{rem:simple-permutoassociahedron}
 Alternatively, one could use the same strategy with the \emph{simple permutoassociahedra} from~\cite{baralicSimplePermutoassociahedron2019}, involving yet another equivalent presentation of symmetric monoidal categories. 
\end{rem}
The following proposition establishes relations between the Mac Lane and Kapranov presentations of symmetric monoidal categories.
\begin{proposition}[Kapranov--Mac Lane comparison]
\label{Kapranov-MacLane}
\leavevmode
\begin{enumerate}
\item Let $\phi_1,\phi_2: t_1 \to t_2$  be parallel morphisms of $\mathcal{S}^{\mathrm{K}}$, and suppose that $[\phi_1]^{\mathrm{K}}=[\phi_2]^{\mathrm{K}}$. Then $[\phi_1]^{\mathrm{ML}}=[\phi_2]^{\mathrm{ML}}$.
\item Let $\phi$ be a morphism of $\mathcal{S}^{\mathrm{ML}}$. Then there is a morphism $\psi$ of 
$\mathcal{S}^{\mathrm{K}}$ such that $[\phi]^{\mathrm{ML}}=[\psi]^{\mathrm{ML}}$.
\end{enumerate}
\end{proposition}
\begin{proof}
The proof of Point (1) is visualized in~\cref{fig:dodecagon} (left).  
The two dotted lines delimit two Mac Lane hexagons on the top and at the bottom and a naturality square in the middle.
Explicitly, the two dotted $\tau$-morphisms are $\tau_{a,b\otimes c}$ and $\tau_{a,c\otimes b}$.  
As for Point (2), we observe that a morphism $\psi$ as in the statement can be obtained by repeatedly applying the procedure described by Mac Lane in the quotation which follows \cref{def:free-Kap} above.
\end{proof}
\begin{thm}[Mac Lane coherence theorem for symmetric monoidal categories]
\label{thm:coherence-MacLane}
    For any symmetric monoidal category $\mathcal{C}$, for any  function $\rho : S \to \obj(\mathcal{C})$, and for any two parallel morphisms $\phi_1,\phi_2: t_1 \to t_2$ in~$\mathcal{S}^{\mathrm{ML}}$, we have $\extd{\phi_1}{\mathrm{ML}}=\extd{\phi_2}{\mathrm{ML}}$.
\end{thm}
\begin{proof} 
Since the functor $\extd{-}{\mathrm{ML}}$ factorizes through the functor $[-]^{\mathrm{ML}}$, it is enough to prove that $[\phi_1]^{\mathrm{ML}}=
[\phi_2]^{\mathrm{ML}}$. 
By Point (2) of~\cref{Kapranov-MacLane}, there exist $\psi_1$ and $\psi_2$ in $\mathcal{S}^{\mathrm{K}}$ such that $[\psi_1]^{\mathrm{ML}}=[\phi_1]^{\mathrm{ML}}$ and 
$[\psi_2]^{\mathrm{ML}}=[\phi_2]^{\mathrm{ML}}$.
In particular $\psi_1$ and $\psi_2$ are parallel, so by~\cref{thm:coherence-Kapranov} we get $[\psi_1]^{\mathrm{K}}=[\psi_2]^{\mathrm{K}}$, and by Point (1) of~\cref{Kapranov-MacLane} we have $[\psi_1]^{\mathrm{ML}}=[\psi_2]^{\mathrm{ML}}$.
Thus, we have $[\phi_1]^{\mathrm{ML}} = [\psi_1]^{\mathrm{ML}} =  [\psi_2]^{\mathrm{ML}} =  [\phi_2]^{\mathrm{ML}}$, which concludes the proof.
\end{proof}
\begin{rem} 
\label{rem:Kapranov-to-MacLane}
One can see easily that this proof  also shows that the categories  $\mathcal{F}(S)^{\mathrm{K}}$ and~$\mathcal{F}(S)$ are isomorphic.  
The statement of this fact is unrelated to coherence issues, but  its proof relies on Kapranov style coherence.
In other words, the  proof that Kapranov's conditions imply Mac Lane's conditions  is non-trivial, in contrast to the converse direction (cf. \cref{Kapranov-MacLane}); a result of the magic of polytopes!
\end{rem}
\begin{rem}
\label{MacLane-Kapranov-Simple}
Note that contrary to the case of the operahedra, there does not seem to exist an orientation vector whose induced orientation on the edges of the permuto-associahedra coincides with a consistent orientation of the $\beta$ and $\tau$ arrows \emph{based on conventions independent of the orientation vector}.
This follows from the observation that the dodecagon (\cref{fig:dodecagon}, left) involves $\beta^{-1}$ arrows.  
The same remarks apply to the simple permuto-associahedra of \cite{baralicSimplePermutoassociahedron2019}.
As for the original presentation of Mac Lane (for which no polytopal correspondence is known), one could still hope to have an associated  term rewriting system.  
But instead Mac Lane's proof (rightly!) proceeds in two stages: first using rewriting for the monoidal part ($\beta$ only), and then dealing with the symmetric part using Coxeter's presentation of the symmetric groups. 
It seems that one cannot do better. 
Indeed, even if  Mac Lane's hexagon does not involve $\beta^{-1}$  arrows, the latter would pop up when taking the combinatorics of orientation of the $\tau$ arrows into account. 
As an illustration, suppose that we decide to move parentheses to the right for $\beta$, fix a total order on $S$ and split the involutive $\tau$ into $\tau$ and $\tau^{-1}$ according to where the maximum lies. 
Then, for $\mu<\kappa<\nu$ the hexagon becomes
 \begin{center}
\resizebox{0.35\linewidth}{!}{
    \begin{tikzpicture}[scale=2.5]
    \node (P1) at (0,1) {$(\kappa\otimes \mu) \otimes \nu$};
    \node (P2) at (-0.866,0.5) {$\kappa\otimes (\mu \otimes \nu)$};
    \node (P3) at (-0.866,-0.5) {$(\mu \otimes \nu) \otimes \kappa$};
    \node (P4) at (0,-1) {$\mu \otimes (\nu \otimes \kappa)$};
    \node (P5) at (0.866,0.5) {$(\mu \otimes \kappa) \otimes \nu$} ;
    \node (P6) at (0.866,-0.5) {$\mu \otimes (\kappa \otimes \nu)$};
    \draw[->] (P1)--(P2) node[midway,above left] {$\beta$};
    \draw[<-] (P2)--(P3) node[midway,left] {$\tau$};
    \draw[->] (P3)--(P4) node[midway,below left] {$\beta$};
    \draw[->] (P1)--(P5) node[midway,above right] {$\tau\otimes 1$};
    \draw[->] (P5)--(P6) node[midway,right] {$\beta$};
    \draw[<-] (P6)--(P4) node[midway,below right] {$1\otimes\tau$};
\end{tikzpicture}}
\end{center}
and a local confluence diagram for the pair of rewritings  out of $(\mu\otimes\nu)\otimes \kappa$ cannot be completed without inverting $\beta$ arrows.
\end{rem}

%% file: sec/applications.tex

\section{Perspectives}


\subsection{Further applications} 
\label{sec:further}
One can also use the same strategy to prove coherence for \emph{unital} non-symmetric monoidal categories, using the unital associahedra of F. Muro and A. Tonks \cite{muroUnitalAssociahedra2014}.

It is natural to ask if the construction of unital associahedra could be extended to the permutoassociahedra, in such a way as to provide a topological proof of coherence for unital symmetric monoidal categories. 
The question of the existence of these constructions at the operadic level (i.e. does there exist unital operahedra, symmetric operahedra, and unital symmetric operahedra?) is, to our knowledge, still open as well. 

Another immediate application of \cref{thm:top-coherence} is the coherence of strong non-symmetric monoidal functors between non-symmetric monoidal categories \cite{epsteinFunctorsTensoredCategories1966}. 
The corresponding topological objects are in this case the family of multiplihedra \cite{Stasheff70,Forcey08}.
The generalization to strong morphisms between non-symmetric categorified operads also goes through, involving this time the family of multiploperahedra described at the end of the introduction in \cite{MazuirLA22}.

In the same spirit as in \cref{thm:coherence-operahedra}, one could obtain coherence results for categorifications of many operad-like structures, for instance the ones described in \cite{BMO20}: categorified modular operads, wheeled properads, and permutads (shuffle algebras), among others.
In order to treat cyclic and symmetric structures, one could take inspiration from the reduction process followed in \cite{curienCategorifiedCyclicOperads2020} for the case of cyclic symmetric categorified operads.

\subsection{Higher categories} 
\label{sec:higher}

\cref{thm:top-coherence} shows the precise relationship between coherence and connectedness.
In addition to Kapranov's claim \cite{kapranov1993}, it clarifies other statements in the literature, such as the proof of \cite[Prop.~3.9]{KapranovVoevodsky94}.
There, the incipit ``since $P_n$ is a convex polytope'' could be replaced by a more precise ``since $P_n$ is simply connected''.

In the case of (symmetric) monoidal categories, \cref{thm:top-coherence} demonstrates that coherence is equivalent to the vanishing of the first homotopy groups of the (permuto-)associahedra. 
Since the (permuto-)associahedra are contractible, and therefore all their homotopy groups vanish, one could hope for a topological proof of higher dimensional coherence theorems.

One dimension higher, N. Gurski has shown in \cite[Thms.~22 \& 23]{Gurski11} that coherence for (braided) monoidal bicategories is equivalent to the vanishing of fundamental $2$-groupoids of braid groups.
Recent results of S. Barkan provide evidence for higher dimensional statements, relating coherence diagrams of $\infty$-operads to the connectivity of certain operadic partition complexes \cite{barkanArityApproximationInfty2022}.
It seems likely that the present results could be interpreted as a strict version and a special case of \cite[Thm.~B]{barkanArityApproximationInfty2022}. 
It would be interesting to see how the permuto-associahedra arise in the strictification process, and how they are related to operadic partition complexes.

%% file: Coherence.bbl
\providecommand{\bysame}{\leavevmode\hbox to3em{\hrulefill}\thinspace}
\providecommand{\MR}{\relax\ifhmode\unskip\space\fi MR }
\providecommand{\MRhref}[2]{%
  \href{http://www.ams.org/mathscinet-getitem?mr=#1}{#2}
}
\providecommand{\href}[2]{#2}
\begin{thebibliography}{{Lap}22}

\bibitem[Bar22]{barkanArityApproximationInfty2022}
Shaul Barkan, \emph{Arity {{Approximation}} of $\infty$-{{Operads}}}, ArXiv
  e-prints (2022), \texttt{arXiv:2207.07200}.

\bibitem[BB97]{bestvinaMorseTheoryFiniteness1997}
Mladen Bestvina and Noel Brady, \emph{Morse theory and finiteness properties of
  groups}, Inventiones mathematicae \textbf{129} (1997), no.~3, 445--470.

\bibitem[BIP19]{baralicSimplePermutoassociahedron2019}
Djordje Barali{\'c}, Jelena Ivanovi{\'c}, and Zoran Petri{\'c}, \emph{A simple
  permutoassociahedron}, Discrete Mathematics \textbf{342} (2019), no.~12,
  111591.

\bibitem[BMO23]{BMO20}
Michael Batanin, Martin Markl, and Jovana Obradovi{\'c}, \emph{Minimal models
  for graph-related (hyper)operads}, Journal of Pure and Applied Algebra
  \textbf{227} (2023), no.~7, 107329.

\bibitem[BN98]{baaderTermRewritingAll1998}
Franz Baader and Tobias Nipkow, \emph{Term {{Rewriting}} and {{All That}}},
  {Cambridge University Press}, {Cambridge}, 1998.

\bibitem[Bro06]{Brown2006}
Ronald Brown, \emph{Topology and groupoids}, {Booksurge Publishing}, 2006.

\bibitem[CL23]{CastilloLiu21}
Federico Castillo and Fu~Liu, \emph{The permuto-associahedron revisited},
  European Journal of Combinatorics \textbf{110} (2023), 103706.

\bibitem[CLA24]{CLA24}
Pierre-Louis Curien and Guillaume Laplante-Anfossi, \emph{Term rewriting on
  nestohedra}, ArXiv e-prints (2024), \texttt{arXiv:2403.15987}.

\bibitem[CO20]{curienCategorifiedCyclicOperads2020}
Pierre-Louis Curien and Jovana Obradovi{\'c}, \emph{Categorified {{Cyclic
  Operads}}}, Applied Categorical Structures \textbf{28} (2020), no.~1,
  59--112.

\bibitem[COI19]{curienSyntacticAspectsHypergraph2019a}
Pierre-Louis Curien, Jovana Obradovi{\'c}, and Jelena Ivanovi{\'c},
  \emph{Syntactic aspects of hypergraph polytopes}, Journal of Homotopy and
  Related Structures \textbf{14} (2019), no.~1, 235--279.

\bibitem[DP11]{DP-HP}
Kosta Do{\v s}en and Zoran Petri{\'c}, \emph{Hypergraph polytopes}, Topology
  and its Applications \textbf{158} (2011), 1405--1444.

\bibitem[DP15]{DP15}
\bysame, \emph{Weak cat-operads}, Logical Methods in Computer Science
  \textbf{11} (2015), no.~1, 1--23.

\bibitem[DS24]{DefanSack24}
Colin Defant and Andrew Sack, \emph{Operahedron lattices}, ArXiv e-prints
  (2024), \texttt{arXiv:2402.12717}.

\bibitem[Eps66]{epsteinFunctorsTensoredCategories1966}
D.~B.~A. Epstein, \emph{Functors between tensored categories}, Inventiones
  mathematicae \textbf{1} (1966), no.~3, 221--228.

\bibitem[For08]{Forcey08}
Stefan Forcey, \emph{Convex hull realizations of the multiplihedra}, Topology
  Appl. \textbf{156} (2008), no.~2, 326--347.

\bibitem[Gur11]{Gurski11}
Nick Gurski, \emph{Loop spaces, and coherence for monoidal and braided monoidal
  bicategories}, Advances in Mathematics \textbf{226} (2011), no.~5,
  4225--4265.

\bibitem[Kap93]{kapranov1993}
Mikhail~M. Kapranov, \emph{The permutoassociahedron, {{Mac Lane}}'s coherence
  theorem and asymptotic zones for the {{KZ}} equation}, Journal of Pure and
  Applied Algebra \textbf{85} (1993), no.~2, 119--142.

\bibitem[KV94]{KapranovVoevodsky94}
M.~M. Kapranov and V.~A. Voevodsky, \emph{2-categories and {Z}amolodchikov
  tetrahedra equations}, Proc. Sympos. Pure Math., vol.~56,
  p.~177{\textendash}259, Amer. Math. Soc., Providence, RI, 1994.

\bibitem[{Lap}22]{laplante-anfossiDiagonalOperahedra2022a}
Guillaume {Laplante-Anfossi}, \emph{The diagonal of the operahedra}, Advances
  in Mathematics \textbf{405} (2022), 108494.

\bibitem[LM23]{MazuirLA22}
Guillaume {Laplante-Anfossi} and Thibaut Mazuir, \emph{The diagonal of the
  multiplihedra and the tensor product of {A}-infinity morphisms}, Journal de
  l'{\'E}cole polytechnique --- Math{\'e}matiques \textbf{10} (2023), 405--446.

\bibitem[LV12]{LodayVallette12}
Jean-Louis Loday and Bruno Vallette, \emph{Algebraic operads}, Grundlehren Der
  Mathematischen Wissenschaften [{{Fundamental}} Principles of Mathematical
  Sciences], vol. 346, {Springer-Verlag}, {Berlin}, 2012.

\bibitem[ML63]{MacLane63}
Saunders Mac~Lane, \emph{Natural associativity and commutativity}, Rice Univ.
  Studies \textbf{49} (1963), no.~4, 28--46.

\bibitem[MT14]{muroUnitalAssociahedra2014}
Fernando Muro and Andrew Tonks, \emph{Unital associahedra}, Forum Mathematicum
  \textbf{26} (2014), no.~2, 593--620.

\bibitem[RZ94]{reinerCoxeterassociahedra1994}
Victor Reiner and G{\"u}nter~M. Ziegler, \emph{Coxeter-associahedra},
  Mathematika \textbf{41} (1994), no.~2, 364--393.

\bibitem[Sta70]{Stasheff70}
James~D. Stasheff, \emph{{$H$}-spaces from a homotopy point of view}, Lecture
  Notes in Mathematics, Vol. 161, Springer-Verlag, Berlin, 1970.

\bibitem[VdL03]{VanDerLaan03}
Pepijn Van~der Laan, \emph{Coloured {K}oszul duality and strongly homotopy
  operads}, ArXiv e-prints (2003), \texttt{arXiv:0312147}.

\bibitem[Zie95]{Ziegler95}
G\"{u}nter~M. Ziegler, \emph{Lectures on polytopes}, Graduate Texts in
  Mathematics, vol. 152, Springer-Verlag, New York, 1995.

\end{thebibliography}
